\documentclass[12pt]{scrarticle}
\usepackage{amsmath,amssymb,amsfonts,amsthm}
\usepackage{tikz,tikz-cd}
\usepackage{url}
\usepackage[colorlinks=true]{hyperref} 
\usepackage{enumitem}

\hypersetup{urlcolor=blue,linkcolor=black,citecolor=black,colorlinks=true}
\usepackage[symbol=$\uparrow~$,numberlinked=true]{footnotebackref}
\usepackage[all]{xy}
\usepackage{verbatim}
\usepackage{soul}

\theoremstyle{plain} % the usual style for theorems
\newtheorem{Theorem}{Theorem}[section]
\newtheorem{Lemma}[Theorem]{Lemma}
\newtheorem{Corollary}[Theorem]{Corollary}
\newtheorem{Proposition}[Theorem]{Proposition}
\newtheorem*{Proposition*}{Proposition}

\theoremstyle{definition}

\newtheorem{Definition}[Theorem]{Definition}
\newtheorem{Example}[Theorem]{Example}

\newtheorem*{Question*}{Question}

\theoremstyle{remark}
\newtheorem{Remark}[Theorem]{Remark}
\newtheorem*{Remark*}{Remark}

\usepackage{accents}%Double hat
\newlength{\dhatheight}

\newcommand{\Spec}{\operatorname{Spec}}

\newcommand{\res}{\operatorname{res}}
\newcommand{\Sp}{\operatorname{Sp}}
\newcommand{\rig}{\operatorname{rig}}

\newcommand{\Hom}{\operatorname{Hom}}
\newcommand{\End}{\operatorname{End}}
\newcommand{\Aut}{\operatorname{Aut}}
\newcommand{\id}{\operatorname{id}}

\newcommand{\Q}{\operatorname{Q}}

\newcommand{\Frob}{\operatorname{Frob}}

\newcommand{\Mat}{\mathsf{Mat}} 
\newcommand{\GL}{\operatorname{GL}}

\newcommand{\Lie}{\operatorname{Lie}}

\newcommand{\KI}{K_{\infty}}
\newcommand{\CI}{\mathbb{C}_{\infty}}

\newcommand{\bF}{\mathbb{F}}
\newcommand{\bG}{\mathbb{G}}

\newcommand{\bP}{\mathbb{P}}

\newcommand{\bT}{\mathbb{T}}

\newcommand{\bZ}{\mathbb{Z}}

\newcommand{\cE}{\mathcal{E}}
\newcommand{\cF}{\mathcal{F}}

\newcommand{\cJ}{\mathcal{J}}

\newcommand{\cL}{\mathcal{L}}
\newcommand{\cM}{\mathcal{M}}

\newcommand{\cO}{\mathcal{O}}

\newcommand{\fa}{\mathfrak{a}}

\newcommand{\fd}{\mathfrak{d}}

\newcommand{\fj}{\mathfrak{j}}

\newcommand{\fm}{\mathfrak{m}}

\newcommand{\fX}{\mathfrak{X}}

\newcommand{\CC}{\mathsf{C}}

\newcommand{\fsf}{\mathfrak{sf}}

\numberwithin{equation}{section}

\title{Residue of special functions of Anderson $A$-modules at the characteristic graph}

\author{Quentin Gazda\thanks{Centre de Mathématiques Laurent Schwartz (CMLS), \'Ecole Polytechnique, Cour Vaneau F-91120 Palaiseau},\\ and Andreas Maurischat\thanks{FH Aachen University of Applied Sciences \& RWTH Aachen University, D-52062 Aachen}}
\date{}

\begin{document}

\maketitle

{\footnotesize
\tableofcontents
}

\section{Introduction}
\subsection*{Context and motivations}
Let $(C,\cO_C)$ be a geometrically irreducible smooth projective curve over a finite field $\bF$ and let $\infty$ be a closed point on it. Let $A=\cO_C(C\setminus \{\infty\})$ be the $\bF$-algebra of rational functions on $C$ that are regular away from $\infty$. Let $K$ be the function field of $C$. We denote by $\KI$ the completed local field of $C$ at the point $\infty$ and we let $\CI$ be the completion of an algebraic closure of $\KI$. \\

This paper deals with the following objects introduced in the pioneer work of Anderson \cite{anderson}.
\begin{Definition}[Anderson $A$-module]\label{def:A-module}
An $A$-module scheme $E$ over $\CI$ is called an \emph{Anderson $A$-module of dimension $d$} if the following holds:
\begin{enumerate}
\item As an $\bF$-vector space scheme over $\CI$, $E$ is isomorphic to $\bG_a^d$ and, 
\item The $L$-linear map $\partial a$, induced by the functorial action of $a\in A$ on the $\CI$-vector space $\Lie_E(\CI)$, satisfies $(\partial a-a\cdot\id)^d=0$. 
\end{enumerate}
We refer to Section \ref{sec:preliminaries} below for the details. 
\end{Definition}

If $C$ is the projective line over $\bF$ and $\infty$ is the point $[0:1]$, then $A$ is identified with the polynomial ring $\bF[t]$. In this situation, the \emph{Carlitz module} is a certain Anderson $A$-module scheme over $K$ playing a role similar to that of the multiplicative group scheme in the arithmetic of function fields (\emph{cf} \cite{gazdajunger}). It consists of the $\bF[t]$-module scheme $\CC$ over $K=\bF(\theta)$ which, as an $\bF$-vector space scheme, is $\bG_a$ and whose $\bF[t]$-module action is given by $t\cdot x:=\theta x+x^q$ for $x\in \bG_a(-)$. It is an Anderson $\bF[t]$-module of dimension one. \\

The Carlitz module is the simplest example of a \emph{Drinfeld module}, themselves encompassed in the definition of Anderson $A$-modules. \\
The arithmetic of Anderson modules is in many aspects comparable to that of abelian varieties, but with the fundamental difference that their coefficients are function rings or fields. Following this analogy, one can study their geometric features. For instance, one may define their corresponding ``cohomological realizations'' such as the analogue of the Betti (\emph{i.e.} singular) and that of the de Rham realization (see \emph{e.g.} \cite{hartl-juschka}).\\

In \cite{gazda-maurischat}, we constructed an $A$-module $\fsf(E)$ of \emph{special functions} attached to an Anderson $A$-module $E$, assembling prior existing construction in the literature. Roughly speaking (see Definition \ref{def:sf} for details), $\fsf(E)$ is of the sub-$A$-module of a certain completed tensor product
\begin{equation}\label{eq:completion}
A\hat{\otimes}_{\bF}E(\CI),
\end{equation}
relative to a certain metric on $E(\CI)$, consisting of elements $\psi$ for which $(1\otimes a)\psi=(a\otimes 1)\psi$ for all $a\in A$. Although the counterpart of the construction \eqref{eq:completion} does not make sense for abelian varieties (due to the lack of a tensor product over $\bF$), the resulting module $\fsf(E)$ is--whenever $E$ is uniformizable--canonically isomorphic to the dual of the \emph{Betti realization} of $E$ (\cite[Thm. 3.3]{gazda-maurischat}). \\

The prototype of such special functions is the \emph{Anderson-Thakur generating function} $\omega$, first introduced by Anderson--Thakur in \cite{andersonthakur} in the situation of the Carlitz module. It is given by the infinite product
\begin{equation}\label{eq:omega}
\omega(t)=\lambda_\theta \prod_{i=0}^{\infty}{\left(1-\frac{t}{\theta^{q^i}}\right)^{-1}}
\end{equation}
which converges in the Tate algebra $\CI\langle t \rangle$, where $\lambda_\theta$ is a fixed $(q-1)$th-root of $-\theta$ in $\CI$. The element $\omega(t)$ generates the free $\bF[t]$-module $\fsf(\CC)$ of special functions attached to $\CC$. \\

The purpose of the elements of $\fsf(E)$ is to play the same role for $E$ as that played by $\omega$ for the Carlitz module. Some fascinating features of $\fsf(E)$ are its relation to periods, function fields Gauss sums (known as \emph{Gauss-Thakur sums}) and, by the very recent work of Hermes Ferraro \cite{ferraro22,ferraro23}, to \emph{Pellarin type $L$-series}.  Special functions permit to extend results known for $\omega$ (universal Gauss-Thakur sums \cite{angles} and Pellarin's identity \cite{pellarin}) from $\CC$ to more general Anderson modules. \\

In \cite[Thm. 3.11]{gazda-maurischat} we also described a relation among $\fsf(E)$ and the \emph{period lattice} $\Lambda_E$ of $E$ which we now recall. Given $u\in A$ a non constant element such that $K/\bF(u)$ is a (finite) separable field extension--we say that $u$ is \emph{separating}--there is an isomorphism of $A$-modules:
\begin{equation}\label{eq:inverse-residue}
\delta_u:\mathfrak{d}_{A/\bF[u]}\otimes_A \Lambda_E \stackrel{\sim}{\longrightarrow} \fsf(E)
\end{equation}
where $\mathfrak{d}_{A/\bF[u]}$ is the different ideal relative to the extension $K/\bF(u)$. In the situation of $A=\bF[t]$, the ideal $\mathfrak{d}_{\bF[t]/\bF[u(t)]}$ is principally generated by the derivative $u'$ of $u$, and the corresponding isomorphism \eqref{eq:inverse-residue} maps the generator $u'(t)\otimes \tilde{\pi}$ to $\omega(t)$. Here, we wrote $\tilde{\pi}$ for the \emph{Carlitz period}, the function field counterpart of $2\pi i$. \\

Although \eqref{eq:inverse-residue} is functorial in $E$, it is non canonical as it depends on the choice of $u$. However, one can make the following observation: the class of $\mathfrak{d}_{A/\bF[u]}$ and that of $\Omega^1_{A/\bF}$ are mutual inverses to each other in the group $\operatorname{Pic}(A)$. It then becomes natural to ask whether \eqref{eq:inverse-residue} admits a ``canonical inverse'', taking the form
\begin{equation}\label{eq:residue}
\Omega^1_{A/\bF}\otimes_A  \fsf(E) \stackrel{\sim}{\longrightarrow} \Lambda_E.
\end{equation}
That the above does not depend on $u$ was already noticed by the authors in the situation of $A=\bF[t]$ (Proposition 3.13 in \emph{loc.\,cit.}).\\

The purpose of this text is to provide a positive answer to the above question, and to interpret \eqref{eq:residue} as a residue map. In this respect, we  extend the result of the second author in \cite[Prop.~3.8]{maurischat} to the general coefficient ring $A$. 

\subsection*{Constructions and results}
The main technicality of this note rests in the signification of \emph{what we mean} by a residue map. Let us fix some notations in order to be more precise: let $E$ be an Anderson $A$-module of dimension $d$ over~$\CI$. To pursue the presentation, we shall state precisely the construction of the completed module \eqref{eq:completion}: once an isomorphism $\kappa:E\cong \bG_a^d$ of $\bF$-vector spaces $\CI$-schemes is fixed--called \emph{a choice of coordinates}--we obtain a non-canonical identification among $A\otimes E(\CI)$ and $(A\otimes \CI)^d$. The Gauss norm on the latter induces a metric on the former, and one shows that this metric is independent of the choice of coordinates $\kappa$. Taking the completion with respect to this metric, we obtain \eqref{eq:completion}, interpreting it thusly as a module of ``$E(\CI)$-valued functions on the affinoid disk $\operatorname{Sp}(A\hat{\otimes}\CI)$''.\\

Let $\mathfrak{X}$ be the rigid analytic variety $(\Spec A\otimes \CI)^{\text{rig}}$; we shall review its construction, as well as sheaves of meromorphic functions on it, in Section \ref{sec:rigid-analytic-geometry} below. If one wishes to extend \eqref{eq:completion} to more general admissible opens $\mathfrak{U}$ of $\mathfrak{X}$, one encounters the following issue: surely, one wants ``$E(\CI)$-valued functions on $\mathfrak{U}$'' to be isomorphic to $\cO_{\mathfrak{X}}(\mathfrak{U})^d$ after the choice of $\kappa$. Yet, this is not a reasonable wish in case the dimension is bigger than one. Indeed, assume $d>1$ and suppose we have two choices of coordinates $\kappa$ and $\kappa':E\cong \bG_a^d$ which differ from one another by the automorphism
\[
\begin{pmatrix}
1 & & & \Frob_{\bF} \\
 & 1 &  & \\
 & & \ddots & \\
 & & & 1
\end{pmatrix}\in \Aut(\bG_a^d).
\]
Then, the module of ``$E(\CI)$-valued functions on $\mathfrak{U}$'' will also be identified with $\cO_{\mathfrak{X}}(\mathfrak{U}\cap \Frob_\bF(\mathfrak{U}))^d$, preventing the notion of `` $E(\CI)$-valued functions on $\mathfrak{U}$'' to be well-defined. \\

To make things independent on the choice of coordinates, we are required to work with admissible opens $\mathfrak{U}$ that we call \emph{$\tau$-costable}, meaning that $\mathfrak{U}\subseteq \Frob_{\bF}(\mathfrak{U})$. If $\underline{\mathfrak{X}}=(\operatorname{Cat}\mathfrak{X},\operatorname{Cov}\mathfrak{X})$ denotes the site on~$\mathfrak{X}$ corresponding to the \emph{strong Grothendieck topology}, we denote by $\,^\tau\underline{\mathfrak{X}}=(\operatorname{Cat}^\tau\mathfrak{X},\operatorname{Cov}^\tau\mathfrak{X})$ the induced site compound of $\tau$-costable admissible opens and coverings of $\underline{\mathfrak{X}}$ (we refer the reader to \emph{e.g.} \cite[\S 5.1]{bosch} for notions on Grothendieck topologies, which only appear at a basic level in this text). The \emph{$\tau$-costable site} $\,^\tau\underline{\mathfrak{X}}$ happens to be enough for our purpose (see Section \ref{sec:AoA}):

\begin{Proposition}\label{prop:intro-E-valuedfunctions}
Up to canonical isomorphisms, there exists a unique pair $(\mathcal{E},\iota)$ consisting of a sheaf $\mathcal{E}$ of $A\otimes A$-modules on $\,^{\tau}\underline{\mathfrak{X}}$ and a family $\iota=(\iota(\mathfrak{U}):A\otimes E(\CI)\to \mathcal{E}(\mathfrak{U}))_{\mathfrak{U}}$ of $A\otimes A$-linear morphisms indexed by objects $\mathfrak{U}$ of $\operatorname{Cat}^{\tau}\mathfrak{X}$ such that, for any choice of coordinates $\kappa$ and any such $\mathfrak{U}$, there exists an $A$-linear isomorphism $\kappa(\mathfrak{U}):\mathcal{E}(\mathfrak{U})\stackrel{\sim}{\to} \cO_{\mathfrak{X}}(\mathfrak{U})^d$ inserting in a pushout square of $A$-modules:
 \begin{equation}\label{eq:intro-diagram-E-functions}
\begin{tikzcd}
(A\otimes \CI)^d \arrow[r,"\kappa^{-1}","\sim"']\arrow[d,hookrightarrow] & A\otimes E(\CI) \arrow[d,"\iota(\mathfrak{U})"] \\
 \cO_\mathfrak{X}(\mathfrak{U})^d \arrow[r,"\kappa(\mathfrak{U})^{-1}"] & \mathcal{E}(\mathfrak{U})
\end{tikzcd}
\end{equation}
\end{Proposition}

We call $\mathcal{E}(\mathfrak{U})$ the \emph{$A\otimes A$-module of $E(\CI)$-valued holomorphic functions on $\mathfrak{U}$}. More generally, if $\cM_\mathfrak{X}$ denotes the sheaf of meromorphic functions on $\underline{\mathfrak{X}}$ and $\cL\subset \cM_\mathfrak{X}$ is a \emph{$\tau$-costable} sub-$\cO_\mathfrak{X}$-modules sheaf (Definition \ref{def:costable}), there exists a unique sheaf $\mathcal{E}(\cL)$ of $A\otimes A$-modules on $\,^\tau\underline{\mathfrak{X}}$ verifying Proposition~\ref{prop:intro-E-valuedfunctions} with $\cL$ in place of $\cO_\mathfrak{X}$ in diagram \eqref{eq:intro-diagram-E-functions}. In particular, we have $\mathcal{E}=\mathcal{E}(\cO_\mathfrak{X})$. \\

If $\mathfrak{U}$ is the affinoid disk $\mathfrak{D}:=\operatorname{Sp}(A\hat{\otimes}\CI)$, then $\mathcal{E}(\mathfrak{D})$ is canonically isomorphic to $A\hat{\otimes} E(\CI)$ (Lemma \ref{lem:E-on-disk}). Therefore, the sheaf $\mathcal{E}$ naturally extends our construction \eqref{eq:completion} to the whole $\,^\tau\underline{\mathfrak{X}}$. We subsequently generalize the definition of the module of special functions (Definition 3.1 in \emph{loc.\,cit.}) as follows (Definition~\ref{def:sf}).

\begin{Definition}
Let $\mathfrak{U}$ be an object of $\operatorname{Cat}^{\tau}\mathfrak{X}$. Given $\cL$ a $\tau$-costable sub-$\cO_\mathfrak{X}$-module sheaf of $\cM_\mathfrak{X}$, we let $\fsf_E(\cL)(\mathfrak{U})$ be the sub-$A$-module of $\mathcal{E}(\cL)(\mathfrak{U})$ consisting of elements $\omega$ satisfying $(a\otimes 1)\omega=(1\otimes a)\omega$ for all $a\in A$. 
\end{Definition}

One verifies that $\fsf_E(\cL)$ is a sheaf of $A$-modules on $\,^\tau\underline{\mathfrak{X}}$ (Proposition \ref{prop:sf-sheaf}) which extends $\fsf(E)$ on the disk $\mathfrak{D}$ (Lemma \ref{lem:E-on-disk}). \\

Let $\fj\in \operatorname{Sp}(A\otimes \CI)$ be the prime ideal given by the kernel of the multiplication map $A\otimes \CI\to \CI$. It has been observed in several situations that special functions not only live in the Tate algebra, but can also be analytically continued to meromorphic functions on the full $\mathfrak{X}$ with only poles at $\fj$ and its iterates by $\tau$ (denoted $\tau\fj$, $\tau^2\fj$, etc.) and of order at most $d=\dim E$. For instance, the inverse of the Anderson-Thakur function $\omega(t)$ has infinite radius of convergence and vanishes only at $t=\theta$, $\theta^q$, ... In much greater generality, this also appears in the work of Hartl--Juschka \cite[\S 2.3.4]{hartl-juschka}.\\ 
In our setting, the statement takes the following form (Proposition \ref{thm:continuation-sf}):
\begin{Theorem}[Meromorphic continuation of special functions]\label{thm:intro-continuation}
Let $J$ be the rigid analytic divisor $\fj+\tau\fj+\tau^2\fj+...$ on $\mathfrak{X}$. Then, the sheaf $\cO_\mathfrak{X}(d\cdot J)$ is $\tau$-costable and the map
\[
\fsf_E(\cO_\mathfrak{X}(d\cdot J))(\mathfrak{X})\to \fsf_E(\cO_\mathfrak{X}(d\cdot J))(\mathfrak{D})=\fsf_E(\cO_\mathfrak{X})(\mathfrak{D})=\fsf(E),
\]
where the first map is given by restriction, is an $A$-linear isomorphism.
\end{Theorem}

To state the main theorem of this note, we require the notion of \emph{residue map at $\fj$} for $E(\CI)$-valued meromorphic functions. Denote by $\mathcal{J}$ the sheaf of meromorphic functions on $\underline{\mathfrak{X}}$ that are holomorphic away from the support of $J$. Let $\kappa:E\stackrel{\sim}{\to}\bG_a^d$ be a choice of coordinates; note that $\Lie_{\kappa}$ induces coordinates on the tangent space $\Lie_E\cong \bG_a^d$. The residue map at $\fj$ is defined as follows: it is the unique map $\res_\fj$ of $A$-modules making the following diagram commute:
\begin{equation}
\begin{tikzcd}
\Omega^1_{A/\bF}\otimes_{A}\mathcal{E}(\mathcal{J})(\mathfrak{X})\arrow[d,"\kappa(\mathfrak{X})"',"\wr"]\arrow[r,dashed,"\res_\fj"] & \Lie_E(\CI) \arrow[d,"\wr"',"\Lie_\kappa"] \\
\Omega_{A/\bF}^1\otimes_A \cJ(\mathfrak{X})^d  \arrow[r,"\operatorname{residue}_{\fj}"] & \CI^d
\end{tikzcd}
\end{equation}
The map denoted $\operatorname{residue}_{\fj}$ above denotes the residue map for rigid meromorphic functions; its definition is recalled below in Section \ref{sec:continuation-residue}.

Our main result is as follows:
\begin{Theorem}\label{thm:residue}
The morphism $\res_\fj:\Omega^1_{A/\bF}\otimes_{A}\mathcal{E}(\mathcal{J})(\mathfrak{X})\longrightarrow \Lie_E(\CI)$ is independent on the choice of coordinates $\kappa$. Further, through $\fsf(E)=\fsf_E(\cO_\mathfrak{X}(d\cdot J))(\mathfrak{X})\subseteq \mathcal{E}(\mathcal{J})(\mathfrak{X})$, the map $\res_\fj$ induces an isomorphism $\Omega^1_{A/\bF} \otimes_A \fsf(E) \stackrel{\sim}{\to} \Lambda_E$. 
\end{Theorem}

\begin{Remark}
Surprisingly, we shall never use in this note the assumption that $E$ is uniformizable nor abelian. Theorem \ref{thm:residue} hence provides a general method to reconstruct the period lattice of $E$ without any reference to its exponential function. 
\end{Remark}

An immediate consequence of Theorem \ref{thm:residue} is that there are no non-zero special functions that are holomorphic on the whole $\mathfrak{X}$:
\begin{Corollary}\label{cor:no-pole-no-function}
The module $\fsf_E(\mathcal{O}_\mathfrak{X})(\mathfrak{X})$ is trivial.
\end{Corollary}

Compiling the results of this note, we notice that the $A$-module $\fsf(E)$ carries a natural increasing finite filtration indexed by the order of poles along the divisor $J$:
\begin{equation}\label{eq:mysterious-filtration}
(0)\stackrel{\text{Cor~}\ref{cor:no-pole-no-function}}{=} \fsf_E(\cO_\mathfrak{X})(\mathfrak{X}) \hookrightarrow \fsf_E(\cO_\mathfrak{X}(J))(\mathfrak{X})\hookrightarrow \cdots \hookrightarrow \fsf_E(\cO_\mathfrak{X}(d\cdot J))(\mathfrak{X})\stackrel{\text{Thm~}\ref{thm:intro-continuation}}{=} \fsf(E).
\end{equation}
The above has by design at most $d$ jumps. We computed some examples of this filtration in Subsection \ref{subsec:filtration}. We leave the following question open for future research:
\begin{Question*}
To what filtration on $\Lambda_E$ does \eqref{eq:mysterious-filtration} correspond?
\end{Question*}

\paragraph{Context and acknowledgments:} The main ideas of this paper were already in the mind of the authors at the time of \cite{gazda-maurischat}. The emergence of the paper owes much to Giacomo Hermes Ferraro and his desire to read a detailed version of the main theorem. In preparing this manuscript, the authors noticed the fundamental need of a \emph{costable site}--an addition to the theory compared to \emph{loc.\,cit.}--in order to state the results. Apart from the costable theory, the fact that periods should appear as the residue at $t=\theta$ of special functions was already in vogue: for instance, this was already noticed by Pellarin \cite[\S 4.2]{pellarin08} for Drinfeld modules in the $\bF[t]$-case; for $A$ being the coordinate ring of an elliptic curve over $\bF$ and $E$ a Drinfeld-Hayes module, an explicit formula was found by Green-Papanikolas \cite[Thm. 4.6]{green-papanikolas}. Parts of this picture were also present at various stages of \cite{hartl-juschka}, though in a different language assuming particular coefficient rings, abelianess and uniformizability.

The first author is grateful to Max Planck Institute for Mathematics in Bonn for its
hospitality and financial support.

\section{Preliminaries}\label{sec:preliminaries}
\subsection*{Notations}
Let $\bF$ be a finite field of characteristic $p$. Let $(C,\cO_C)$ be a geometrically irreducible smooth projective curve over $\bF$ and let $\infty$ be a closed point on it. Let 
\[
A=H^0(C\setminus \{\infty\},\cO_C)
\]
be the $\bF$-algebra of rational functions on $C$ that are regular away from $\infty$. We let $K=\bF(C)$ be the function field of $(C,\cO_C)$ or, equivalently, the fraction field of $A$. We let $K_{\infty}$ be the completion of $K$ with respect to the valuation given by the order of vanishing at the closed point $\infty$,  $\KI^s$ be a fixed separable closure of $\KI$, and $\CI$ be the completion of $\KI^s$.\footnote{It equals the completion of an algebraic closure of $\KI$.} Let $\ell:A\to \CI$ denote the inclusion and call it the \emph{characteristic morphism}. \\

In this text, a \emph{(commutative) group scheme over $\CI$} is a functor $G$ from the category of $\CI$-algebras to that of abelian groups, which is represented by the functor of points of a scheme. If this functor further lands in the category of $B$-modules for some commutative ring $B$, then we say that $G$ is a \emph{$B$-module scheme over $\CI$}. In particular, we will denote by $\bG_a$ the $\bF$-vector space scheme over $\CI$ given by the tautological functor which assigns to a $\CI$-algebra $R$ its underlying $\bF$-vector space
\[
\bG_a:\{\CI-\text{Algebras}\}\longrightarrow \{\bF-\text{vector~spaces}\}, \quad R\longmapsto R.
\]
As a group scheme, it is canonically isomorphic to the additive group scheme over $\CI$, hence the notation. Given a group scheme $G$ over $\CI$, we define its \emph{tangent group scheme} $\Lie_G$ via the assignment 
\[
\Lie_G:\{\CI-\text{Algebras}\}\longrightarrow \{\bF-\text{vector~spaces}\}, \quad R\longmapsto \ker G\left(R[\varepsilon] \xrightarrow{\varepsilon\mapsto 0} R\right),
\]
where $R[\varepsilon]$ denotes the $R$-algebra of \emph{dual numbers} \cite[\S 12]{milne}. Observe that if $G$ is a $B$-module scheme, then so is $\Lie_G$ by making $b\in B$ act as $\varepsilon\mapsto b\varepsilon$. We denote this action by $\partial b$. \\

Let $q=|\bF|$. By $\CI\{\tau\}$, we mean the non-commutative ring of polynomials over $\CI$ in a variable~$\tau$ subject to the commutation rule $\tau\cdot c=c^q\tau$ for all $c\in \CI$. There is a canonical isomorphism of rings (\emph{cf} \cite[\S 15.d]{milne}):
\[
\CI\{\tau\}\stackrel{\sim}{\longrightarrow} \End_{\bF\text{-vs}/\CI}(\bG_a),\quad \sum_{i}{c_i\tau^i} \longmapsto \sum_{i}{c_i\cdot \Frob_{\bG_a}^{q^i}}.
\]

We may now restate Definition \ref{def:A-module} of the introduction. By an \emph{Anderson $A$-module of dimension $d$ over $\CI$} we mean an $A$-module scheme $E$ over $\CI$ which satisfies:
\begin{enumerate}[label=$(\roman*)$]
    \item As an $\bF$-vector space scheme over $\CI$, $E$ is isomorphic to $\bG_a^d$;
    \item Let $\varphi:A\to \End_{\bF\text{-vs}/\CI}(E)$ be the ring homomorphism induced by the $A$-module structure on $E$ seen as an $\bF$-vector space scheme, and let $\partial\varphi:A\to \End_{\CI}(\Lie_E(\CI))$ be the induced action on the tangent group scheme. Then, for all $a\in A$, $\partial\varphi(a)-\ell(a)$ is nilpotent on $\Lie_E(\CI)$.
\end{enumerate}
We will refer to an isomorphism $\kappa:E\stackrel{\sim}{\to}\bG_a^d$ of $\bF$-vector spaces schemes over $\CI$ as \emph{a choice of coordinates for $E$}. Note that any such choice of coordinates induces an isomorphism ${\Lie_\kappa:\Lie_E \stackrel{\sim}{\to} \bG_a^d}$ of $\bF$-vector space schemes since $\Lie_{\bG_a^d}\cong \bG_a^d$ canonically. 

\subsection*{The characteristic graph ideal $\fj$}
We again denote by $\tau:A\otimes \CI\to A\otimes \CI$ the ring endomorphism which acts $A$-linearly but as the $q$-Frobenius on $\CI$. We let $\fj$ be the \emph{characteristic graph ideal}, \emph{i.e.} the kernel of the multiplication map:
\[
\fj:=\ker(A\otimes \CI\xrightarrow{\ell\otimes \id} \CI).
\]
It is a maximal ideal of $A\otimes \CI$ whose support $V(\fj)$ is the graph of $\ell$. We denote by $\tau^i\fj$ the $i$th iterate of $\fj$ by $\tau$; namely $\tau^i\fj$ is the kernel of the map $A\otimes \CI\to \CI$, $a\otimes b\mapsto \ell(a)b^{q^i}$. In this section we recall some classical facts about $\fj$. \\

We call \emph{separating} any element $u$ in $A$ such that the extension $K/\bF(u)$ is finite and separable.
\begin{Lemma}\label{lemma-p-power}
Any non separating element in $A$ is the $p$-th power of an element in $A$. In particular, for any $a\in A\setminus \bF$, there exists a separating element $u\in A$ and a non-negative integer $k$ such that $a=u^{p^k}$.  
\end{Lemma}

\begin{proof}
Let $a\in A$ be non separating. If $a$ is constant, it clearly is a $p$-th power, hence we assume that the extension $K/\bF(a)$ is finite.  There is an intermediate extension $L$ such that $K/L$ is purely inseparable of degree $p^k$ for some $k\geq 1$, and $L/\bF(a)$ is separable. As $K^{p^k}\subseteq L$, it is in fact an equality by comparing degrees. Hence, $\bF(a)\subset K^p$ and there exists $b\in K$ such that $b^p=a$. As the divisor of $a$ is the the $p$-th power of the divisor of $b$, we obtain that $b\in A$.
\end{proof}

\begin{Lemma}
The ideal $\fj$ is generated by the differences $u\otimes 1-1\otimes \ell(u)$ where $u$ runs over separating elements in $A$.
\end{Lemma}
\begin{proof}
We first claim that $\fj$ is generated by the set $\{a\otimes 1-1\otimes \ell(a)|a\in A\}$. It is clear that $a\otimes 1-1\otimes \ell(a)\in \fj$ for all $a\in A$. Conversely, let $\mathfrak{g}$ be the ideal of $A\otimes \CI$ generated by  $\{a\otimes 1-1\otimes \ell(a)|a\in A\}$. Every element $x\in A\otimes \CI$ can be written as a sum of elementary tensors $a_i\otimes c_i$ for some $a_i\in A$, $c_i\in \CI$. We have
\[
x=\sum_i{a_i\otimes c_i}\equiv \sum_i{1\otimes \ell(a_i)c_i} \pmod{\mathfrak{g}}.
\]
If $x\in \fj$, the right-hand sum is zero, hence $x\in \mathfrak{g}$.

Now, by Lemma \ref{lemma-p-power}, for any $a\in A\setminus \bF$, there exists a separating element $u$ and a positive integer $k$ such that $a\otimes 1-1\otimes \ell(a)=(u\otimes 1-1\otimes \ell(u))^{p^k}$. Thus $\fj\subset \langle\{u\otimes 1-1\otimes \ell(u)~|~ u~\text{separating}\} \rangle$. 
\end{proof}

\begin{Lemma}
If $u$ is a separating element, then $u\otimes 1-1\otimes \ell(u)$ is a uniformizing parameter for~$\fj$.
\end{Lemma}
\begin{proof}
Recall by definition of the module of K\"ahler differentials that we have an isomorphism $\Omega^1_{A/\bF}\otimes \CI\stackrel{\sim}{\longrightarrow} \fj/\fj^2$, $da\mapsto a\otimes 1-1\otimes \ell(a)$. Thus, to prove the lemma, it suffices to show that $du\neq 0$ for $u$ separating. Yet, applying the relative cotangent sequence to $\bF\to \bF(u)\to K$ (\emph{e.g.} \cite[II.10.(26.H)]{matsumara}), we obtain an exact sequence
\[
Kdu \longrightarrow \Omega^1_{K/\bF} \longrightarrow \Omega^1_{K/\bF(u)}
\]
whose right-hand side term is zero as $K/\bF(u)$ is finite separable. Hence, the first map is surjective and the image of $du$ is nonzero.
\end{proof}

\begin{Lemma}\label{lem:m-dist-j}
Let $\fm$ be a maximal ideal of $A\otimes \CI$ different from $\fj$, $\tau\fj$, $\tau^2\fj$, ... There exists a separating element $u$ such that, for all non-negative integer $i$, $u\otimes 1-1\otimes \ell(u)^{q^i}$ does not belong to $\fm$.
\end{Lemma}
\begin{proof}
Let $t$ be a separating element. We first compute the prime ideal decomposition of $(t\otimes 1-1\otimes \ell(t))$ in the Dedekind domain $A\otimes \CI$. The inclusion of Dedekind rings $\bF[t]\otimes \CI\subset A\otimes \CI$ makes $A\otimes \CI$ a free $\bF[t]\otimes \CI$-module of rank $[K:\bF(t)]$. In particular, there are at most $[K:\bF(t)]$ prime divisors of $(t\otimes 1-1\otimes \ell(t))$. For $\sigma:\ell(K)\to \CI$ an $\bF(t)$-algebra morphism, the ideal $\fj^{\sigma}$ of $A\otimes \CI$ generated by the set $\{a\otimes 1-1\otimes \sigma(\ell(a))|a\in A\}$ is maximal and divides the principal ideal $(t\otimes 1-1\otimes \ell(t))$. There are $\#\Hom_{\bF(t)}(\ell(K),\CI)=[K:\bF(t)]$ such ideals, hence 
\begin{equation}
(t\otimes 1-1\otimes \ell(t))=\prod_{\sigma}\fj^{\sigma} \nonumber
\end{equation}
where the product runs over $\sigma\in \Hom_{\bF(t)}(\ell(K),\CI)$. 

We turn to the proof of the lemma. Assume the converse, that is, for all separating elements $v$ there exists $j\geq 0$ such that $v\otimes 1-1\otimes \ell(v)^{q^j}\in \fm$. This means that there exists a non-negative integer $i$ for which $\fm \supset (t\otimes 1-1\otimes \ell(t)^{q^i})=\tau^i\left(\prod_{\sigma}{\fj^{\sigma}}\right)$. By uniqueness of the prime ideal decomposition, there exists $\sigma\in \Hom_{\ell(\bF(t))}(\ell(K),\CI)$ such that $\fm=\tau^i(\fj^{\sigma})$. Because $\fm$ is distinct from $\fj$, $\tau\fj$, $\tau^2\fj$, ..., the morphism $\sigma$ is not the inclusion $\ell(K)\subset \CI$. Because $K$ is generated by separating elements over $\bF$, there exists a separating element $u$ such that $\sigma(\ell(u))\neq \ell(u)$. From our converse assumption, there exists a non-negative integer $j$ such that $u\otimes 1-1\otimes \ell(u)^{q^j}\in \fm=\tau^i(\fj^{\sigma})$. Hence, both $u\otimes 1-1\otimes \ell(u)^{q^j}$ and $u\otimes 1-1\otimes \sigma(\ell(u))^{q^i}$ are in $\fm$. Since $\fm\neq A\otimes \CI$, this implies $\sigma(\ell(u))^{q^i}=\ell(u)^{q^j}$. 

This is a contradiction. Indeed, $\ell(u)^{q^i}$ and $\sigma(\ell(u))^{q^i}=\ell(u)^{q^j}$ have the same minimal polynomial over $\ell(\bF(t))$ so that either the latter polynomial has coefficients in $\bF$ or $i=j$. The first option is impossible as it would imply $\ell(u)\in \bar{\bF}\cap \ell(A)=\bF$. The second option is also impossible as we chose $u$ such that $\sigma(\ell(u))\neq \ell(u)$.   
\end{proof}   

\section{Objects from rigid analytic geometry}\label{sec:rigid-analytic-geometry}
Let $X:=\Spec A\otimes \CI$. In this section, we develop in more details the definition of the rigid analytic variety $\mathfrak{X}:=X^{\text{rig}}$, obtained by applying the rigid analytic GAGA on $X$ (see \cite[\S 5.4]{bosch}). 

\subsection*{The rigid analytic space $\mathfrak{X}$}
To begin with, we define Gauss norms on $A\otimes \CI$. Let $c\in \CI^\times$ and let $\rho:=|c|>0$. For $f\in A\otimes \CI$, we set:
\[
\|f\|_\rho:=\inf \left(\max_i\{|c_i|\rho^{\deg a_i}\}\right)
\]
where the infimum is taken over all representations of $f$ as finite sums $\sum_{i}{a_i\otimes c_i}$ in $A\otimes \CI$.

\begin{Remark}
For $\rho=1$, $\|\cdot\|_1$ was denoted by $\|\cdot \|$ in \cite{gazda-maurischat} and the completed ring denoted by $\bT$. We will rather use the notation $\CI\langle A \rangle$ instead of $\bT$ below. 
\end{Remark}

We prove next that $\|\cdot \|_\rho$ is indeed a Gauss norm. For this, we construct a presentation of $A$ as follows. First observe that $A$ has a canonical filtration by finite dimensional sub-$\bF$-vectors spaces $D_n$ of elements of degree $\leq n$. Clearly $D_0=\bF$ and the union of all $D_n$ is $A$. Let $n_1$ be the least positive integer such that the inclusion $D_0\subset D_{n_1}$ is strict, and let $(a_1,...,a_{s_1})$ be a basis of $D_{n_1}$. We let $A_{n_1}$ be the sub-$\bF$-algebra of $A$ generated by $(a_1,...,a_{s_1})$. Proceed by induction to define $n_i$ as the least integer $>n_{i-1}$ such that $D_{n_i}$ is not contained in $A_{n_{i-1}}$, let $(a_{s_1+1},...,a_{s_{i}})$ be a basis of a supplementary subspace of $(A_{n_{i-1}}\cap D_{n_i})$ in $D_{n_i}$, and defined $A_{n_i}$ as the sub-$\bF$-algebra of $A$ generated by $(a_1,...,a_{s_i})$.

This procedure eventually terminates by the Riemann-Roch Theorem, and we obtain a finite generating family $(a_1,...,a_s)$ of $A$. We obtain a morphism
\begin{equation}\label{eq:good-presentation}
\bF[t_1,...,t_s]\longrightarrow A, \quad t_i\longmapsto a_i 
\end{equation}
and a new degree $d$ given as $d(t_i):=\deg(a_i)$ on monomials of $\bF[t_1,...,t_s]$. By design, \eqref{eq:good-presentation} bears the property that $\deg P(a_1,...,a_s)=d(P(t_1,...,t_s))$ for any $P(t_1,...,t_s)\in \bF[t_1,...,t_s]$.
We denote by $\fa$ the kernel of \eqref{eq:good-presentation}, and we let $(f_1(t_1,...,t_s),...,f_t(t_1,...,t_s))$ be a family of generators of $\fa$. We denote by $\fa_c\subset \CI[t_1,...,t_s]$ the \emph{rescaled} ideal
\[
\left(f_1\left(c^{\deg t_1}t_1,\cdots,c^{\deg t_s}t_s\right),...,f_g\left(c^{\deg t_1}t_1,\cdots ,c^{\deg t_1}t_s\right)\right) \subset \CI[t_1,...,t_s].
\]
 
\begin{Proposition}\label{prop:affinoid}
Through the above presentation of $A$, the ring homomorphism
\[
\iota:\CI[t_1,...,t_s]/\fa_c \to A\otimes \CI, \quad t_i\mapsto a_i\otimes c^{-\deg a_i}
\]
is an isomorphism which makes the Gauss norm on the left-hand side coincide with $\|\cdot \|_\rho$ on the right. 
\end{Proposition}

\begin{Remark}
Recall that the Gauss norm on the quotient $\CI[t_1,...,t_s]/\fa_c$ is defined as
\[
f\in \CI[t_1,...,t_s], \quad \|f\!\pmod{\fa_c}\|:=\inf_{a\in \fa_c} \|f-a\|
\]
where $\|\cdot \|$ is the classical Gauss norm on $\CI[t_1,...,t_s]$; \emph{i.e.} given by the maximum norm of coefficients.
\end{Remark}

\begin{proof}[Proof of Proposition \ref{prop:affinoid}]
That $\iota$ is bijective is clear, so we prove the statement on the norms. Let $x\in \CI[t_1,...,t_s]/\fa_c$. For every $\varepsilon>0$, there exists $\tilde{x}=\sum{\alpha_{i_1,...,i_s}t_1^{i_1}\cdots t^{i_s}_s}\in \CI[t_1,...,t_s]$ lifting $x$ such that $\|x\|_c\geq \|\tilde{x}\|-\varepsilon$. Hence,
\begin{align*}
\|\iota(x)\|_{\rho} &= \left\|\sum_{i_1,...,i_s}{a_1^{i_1}\cdots a_s^{i_s}\otimes \alpha_{i_1,...,i_s}c^{-\deg(a_1^{i_1}\cdots a_s^{i_s})}}\right\|_{\rho} \leq \max_{i_1,...,i_s}|\alpha_{i_1,...,i_s}|=\|\tilde{x}\| \\
&\leq \|x\|_c+\varepsilon.
\end{align*}
Being true for all $\varepsilon>0$, we obtain $\|\iota(x)\|_{\rho}\leq \|x\|_c$. Conversely, let $\iota(x)\in A\otimes \CI$. Similarly, for all $\varepsilon>0$, there exists a representation of $\iota(x)$ as $\sum_i{b_i\otimes c_i}\in A\otimes \CI$ for which
\[
\|\iota(x)\|_\rho\geq \max_i \{|c_i|c^{\deg b_i}\}-\varepsilon.
\]
Let us write $b_i$ as $P_i(a_1,...,a_s)$ for a polynomial $P_i$ with coefficients in $\bF$. Then, 
\[
x=\sum_{i}{c_i P(c^{\deg a_1}t_1,...,c^{\deg a_s}t_s)} \in \CI[t_1,...,t_s]/\fa_c
\]
and $\|x\|_c$ is therefore bounded by $\max_i\{|c_i| \|P_i(c^{\deg a_1}t_1,...,c^{\deg a_s}t_s)\|\}$. Now, by design of the chosen presentation, one has
\[
\|P_i(c^{\deg a_1}t_1,...,c^{\deg a_s}t_s)\|= |c|^{\deg P_i(a_1,...,a_s)},
\]
and we conclude that $\|x\|_c\leq \|\iota(x)\|_\rho+\varepsilon$.
\end{proof}

In virtue of Proposition \ref{prop:affinoid}, we deduce that $\|\cdot \|_\rho$ is a norm, and we let $\CI\langle A\rangle_\rho$ denote the completion of $A\otimes \CI$ with respect to it. Also from Proposition \ref{prop:affinoid}, $\CI\langle A\rangle_\rho$ is an affinoid $\CI$-algebra, hence $\mathfrak{X}_\rho:=\Sp \CI\langle A \rangle_\rho$ is a $\CI$-affinoid variety. If $\rho<\rho'$, there is a canonical ring homomorphism $\CI\langle A \rangle_{\rho'}\to \CI\langle A \rangle_\rho$, and it induces a morphism of affinoid varieties denoted $t_{\rho,\rho'}:\mathfrak{X}_\rho\to \mathfrak{X}_{\rho'}$. 
\begin{Definition}
Let $\mathfrak{X}$ be the rigid analytic variety over $\CI$ resulting as the glueing of $(\mathfrak{X}_\rho)_\rho$ along the morphisms  $(t_{\rho,\rho'})_{\rho,\rho'}$ (for $\rho$, $\rho'$ running over $|\CI^\times|$). 
\end{Definition}

The next lemma is immediate (\emph{e.g.} \cite[\S 5.3~Prop. 5]{bosch}).
\begin{Lemma}\label{lem:admissible}
The family $(\mathfrak{X}_\rho\to \mathfrak{X})_\rho$ is an admissible covering of $\mathfrak{X}$.
\end{Lemma}

By definition, we have $\mathfrak{X}:=\varinjlim_\rho \mathfrak{X}_\rho$ as ringed spaces. In particular, the ring of global sections of $\mathfrak{X}$ is 
\[
\CI\langle \!\langle A\rangle\!\rangle := \varprojlim_\rho \CI\langle A \rangle_\rho.
\]
For $f\in A\otimes \CI$, observe that we have $\|f^{(1)}\|_\rho=\|f\|_{\rho^{1/q}}^q$, so that twisting induces an isomorphism of rings and spaces:
\begin{equation}\label{eq:tau-on-rho}
\tau:\CI\langle A \rangle_\rho\stackrel{\sim}{\longrightarrow} \CI\langle A \rangle_{\rho^{1/q}}, \quad \tau:\mathfrak{X}_{\rho^{1/q}}\stackrel{\sim}{\longrightarrow} \mathfrak{X}_{\rho}.
\end{equation}
Compiling these for all $\rho>0$, we obtain isomorphisms
\begin{equation}
\tau:\CI\langle\!\langle A \rangle\!\rangle\stackrel{\sim}{\longrightarrow} \CI\langle\!\langle A \rangle\!\rangle, \quad \tau:\mathfrak{X}\stackrel{\sim}{\longrightarrow} \mathfrak{X}.
\end{equation}

\begin{Remark}
We refer to \cite[\S 5.4]{bosch} for the fact that $\fX$ is isomorphic to $(\Spec A\otimes \CI)^{\operatorname{rig}}$ as rigid analytic spaces. 
\end{Remark}
\begin{Example}
Let $C=\bP^1_\bF$ and let $\infty$ be the point $[0:1]$. We identify $A$ with $\bF[t]$ and $A\otimes \CI$ with the polynomial ring $\CI[t]$. Denoting $\ell(t)$ by $\theta$, we have $\fj=(t-\theta)$ and, more generally, $\tau^i\fj=(t-\theta^{q^i})$. We have 
\begin{align*}
\CI\langle A \rangle_\rho &\cong \left\{\sum_{n=0}^{\infty}{a_n t^n}\in \CI[\![t]\!]~\bigg |~\lim_{n\to \infty}|a_n|\rho^n \to 0\right\}, \\
\CI\langle\!\langle A \rangle\!\rangle &\cong \left\{\sum_{n=0}^{\infty}{a_n t^n}\in \CI[\![t]\!]~\bigg |~\forall \rho>0:~\lim_{n\to \infty}|a_n|\rho^n \to 0\right\},
\end{align*}
where $\tau$ acts as $\sum_{n=0}^{\infty}{a_n t^n}\mapsto \sum_{n=0}^{\infty}{a_n^q t^n}$. The one-dimensional Tate algebra over $\CI$ is recovered by $\CI\langle A \rangle_1$.
\end{Example}

\subsection*{Sheaves and divisors on $\mathfrak{X}$}

Given a commutative ring $R$ with unit, we denote by $\operatorname{Q}(R)$ its total ring of fractions; \emph{i.e.} the localization of $R$ by the set of non-zero divisors. Given an admissible open $\mathfrak{U}$ in $\mathfrak{X}$, we denote by $\cM_\mathfrak{X}(\mathfrak{U})$ the total ring of fractions $\Q(\cO_\mathfrak{X}(\mathfrak{U}))$. According to Bosch \cite[\S 2]{bosch-meromorph}, the assignment
\[
\mathfrak{U}\longmapsto \cM_\mathfrak{X}(\mathfrak{U})
\]
is a sheaf on $\underline{\mathfrak{X}}$; i.e with respect to the strong Grothendieck topology. We call $\cM_\mathfrak{X}$ the sheaf of meromorphic functions. A \emph{fractional ideal sheaf} on $\underline{\mathfrak{X}}$ is an invertible sub-$\cO_\mathfrak{X}$-module sheaf of $\cM_\mathfrak{X}$. Lacking references, we recall briefly the dictionary among fractional ideal sheaves and rigid analytic divisors. 

For any $x\in |\mathfrak{X}|=|X|$, we have $\cM_{\mathfrak{X},x}=\Q(\cO_{\mathfrak{X},x})$ where $\cO_{\mathfrak{X},x}$ is a discrete valuation ring with residue field $\CI$ which coincides with $\cO_{X,x}$. Given $f\in \mathcal{M}_\mathfrak{X}(\mathfrak{U})$, a non zero meromorphic function on an admissible open $\mathfrak{U}$, we define its divisor as follows. For $x\in \mathfrak{U}$, we let $\operatorname{ord}_x f$ be the order of $f$ in the discrete valued field $\cM_{\mathfrak{X},x}$. The \emph{divisor of $f$ on $\mathfrak{U}$} is the formal (possibly infinite) sum: 
\[
\operatorname{div} f:= \sum_{x\in \mathfrak{U}}{(\operatorname{ord}_x f)\cdot x}. 
\]
Let $D=\sum_{x\in \mathfrak{U}}{n_x\cdot x}$ be a formal $\bZ$-sum of points of $\mathfrak{X}$. We say that $D$ is a \emph{rigid analytic divisor} if there exists an admissible open cover $(\mathfrak{U}_i\stackrel{\iota_i}{\to} \mathfrak{U})_{i\in I}$ such that, for each $i\in I$, there exists $f_i\in \cM_\mathfrak{X}(\mathfrak{U}_i)$ for which
\[
D_{|\mathfrak{U}_i}:=\sum_{x\in \mathfrak{U}_i}{n_{x}\cdot x}=\operatorname{div} f_i.
\]
More generally, we call \emph{divisor} on $\mathfrak{X}$ any formal $\bZ$-sum of points of $\mathfrak{X}$. To any divisor $D$ on $\mathfrak{X}$, one associates a presheaf $\cO_\mathfrak{X}(D)$ given on an admissible open $\mathfrak{U}$ as
\[
\cO_\mathfrak{X}(D)(\mathfrak{U}):=\{g\in \cM_\mathfrak{X}(\mathfrak{U})~|~\operatorname{div} g+D\geq 0\}.
\]
Whenever $D$ is rigid analytic, $\cO_\mathfrak{X}(D)$ is further a fractional ideal sheaf on $\underline{\mathfrak{X}}$. Conversely, to any fractional ideal sheaf $\mathcal{I}$, there is a rigid analytic divisor $D$ such that $\mathcal{I}=\cO_\mathfrak{X}(D)$. Observe in addition that given $f\in \cM_\mathfrak{X}(\mathfrak{X})$ non zero, we have
\[
\cO_\mathfrak{X}(\operatorname{div} f)(\mathfrak{U})=\{g\in \mathcal{M}_\mathfrak{X}(\mathfrak{U})~|~g\cdot f_{|\mathfrak{U}}\in \mathcal{O}_\mathfrak{X}(\mathfrak{U})\}.
\]
Lemma \ref{lem:m-dist-j} has the following counterpart:
\begin{Proposition}\label{prop:O(J)-equals-intersection-O(Ju)}
For $u\in A$ a non constant element, let $\fj_u$ denote the principal ideal $(u\otimes 1-1\otimes \ell(u))$ of $A\otimes \CI$. Also, consider the divisor $J:=\fj+\tau\fj+\tau^2\fj+...$ and $J_u:=\fj_u+\tau\fj_u+\tau^2\fj_u+...$ on $\mathfrak{X}$. Then, both $J$ and $J_u$ are rigid analytic and, for any $n\in \bZ$ and $\mathfrak{U}$ an affinoid subdomain of $\mathfrak{X}$, we have the equality:
\[
\cO_\mathfrak{X}(n\cdot J)(\mathfrak{U})=\bigcap_{\substack{u\in A \\ \text{separating}}} \cO_\mathfrak{X}(n\cdot J_u)(\mathfrak{U})
\]
where the intersection is taken over separating elements $u\in A$. In particular, $\bigcap_{u} \cO_\mathfrak{X}(n\cdot J_u)$ is a sheaf.
\end{Proposition}

\begin{Lemma}\label{lem:inversion-ao1-1oa}
Let $\rho\in |\CI^\times|$ and let $t=\lfloor \log_q \rho\rfloor$. Let $a\in A$ be non constant. Then $a\otimes 1-1\otimes \ell(a)^{q^j}$ is invertible in $\CI\langle A \rangle_\rho$ for all $j>t$. In particular, $\tau^j\fj$ is invertible in $\CI\langle A \rangle_\rho$ for all $j>t$.
\end{Lemma}
\begin{proof}
It suffices to note that, for $j>t$, the series
\[
\frac{-1}{1\otimes \ell(a)^{q^j}}\sum_{n=0}^{\infty}{\left(\frac{a\otimes 1}{1\otimes \ell(a)^{q^j}}\right)^n}
\]
converges in $\CI\langle A \rangle_\rho$ to an inverse of $a\otimes 1-1\otimes \ell(a)^{q^j}$.
\end{proof}

\begin{proof}[Proof of Proposition \ref{prop:O(J)-equals-intersection-O(Ju)}]
Let $\mathfrak{U}$ be an affinoid subdomain of $\mathfrak{X}$. By Lemma \ref{lem:admissible}, there exists $\rho\in |\CI^\times|$ such that $\mathfrak{U}\to \mathfrak{X}$ factors through $\mathfrak{X}_{\rho}\to \mathfrak{X}$. Let $t=\lfloor \log_q \rho\rfloor$. 

On one hand, we have by Lemma \ref{lem:inversion-ao1-1oa},
\begin{equation}\label{eq:former}
\cO_\mathfrak{X}(n\cdot J)(\mathfrak{U})=\{f\in \Q(\cO_\mathfrak{X}(\mathfrak{U}))~|~\forall h\in (\fj\cdot \tau\fj \cdots \tau^t\fj)^n:~hf\in \cO_\mathfrak{X}(\mathfrak{U})\}.
\end{equation}
To prove that $J$ is rigid analytic, observe that the ideal $\fj\cdot \tau\fj \cdots \tau^t \fj$ of $A\otimes \CI$ is locally principal (as $A\otimes \CI$ is a Dedekind domain). There thus exists a Zariski cover $(U_i\to X)_{i\in I}$ such that $\fj\cdot \tau\fj \cdots \tau^t \fj$ is principal on each $U_i$. Rigidifying, we obtain a covering $(U_i^{\operatorname{rig}}\times_{\mathfrak{X}}\mathfrak{U}\to \mathfrak{U})_{i\in I}$ such that $\fj\cdot \tau\fj \cdots \tau^t \fj$ is principal on each admissible open. Hence, $J$ is rigid analytic.

On the other hand, 
\begin{equation}\label{eq:later}
\bigcap_{u} \cO_\mathfrak{X}(n\cdot J_u)(\mathfrak{U})=\left\{f\in \Q(\cO_\mathfrak{X}(\mathfrak{U}))~\bigg |~\forall u:\left(\prod_{j=0}^t{(u\otimes 1-1\otimes \ell(u)^{q^j})}\right)^{\!n}\! f \in \cO_\mathfrak{X}(\mathfrak{U})\right\}.
\end{equation}
Hence $J_u$ is also a rigid analytic divisor. 

We now prove equality among \eqref{eq:former} and \eqref{eq:later}. Clearly, the former is included in the latter. Conversely, let $f=ab^{-1}$ be an element of \eqref{eq:later}, where $a,b\in \cO_\mathfrak{X}(\mathfrak{U})$ with $b$ a non-zero divisor. Let $\fm$ be a maximal ideal in $\cO_\mathfrak{X}(\mathfrak{U})$ for which $b\in \fm$ but $a\notin \fm$. As the rigid analytic GAGA induces a bijection on closed points, $\fm$ arises from a maximal ideal of $A\otimes \CI$ (which we still denote by $\fm$). By assumption, for all $u\in A$ separating,
\[
\left(\prod_{j=0}^t{(u\otimes 1-1\otimes \ell(u)^{q^j})}\right)^{\!n} a \subset \fm.
\]
Since $\fm$ is prime and $a\notin \fm$, there exists $j\in \{0,...,t\}$ for which $u\otimes 1-1\otimes \ell(u)^{q^j}\in \fm$. By Lemma \ref{lem:m-dist-j}, this implies $\fm\in \{\fj, \tau\fj,\tau^2\fj,...\}$. Hence $f$ belongs to \eqref{eq:former}.
\end{proof}

\subsection*{Costable subsets and costable sheaves}
We introduce next the notion of \emph{costability}. Denote by $\tau:\underline{\mathfrak{X}}\to \underline{\mathfrak{X}}$ the morphism of sites given by the $A$-linear Frobenius, and denote by $\tau^{-1}\cM_{\mathfrak{X}}$ the inverse image sheaf.
\begin{Definition}\label{def:costable}
Let $\mathfrak{U}\subset \mathfrak{X}$ be admissible open, $(\mathfrak{U}_i\to \mathfrak{U})_{i\in I}$ an admissible covering.
\begin{enumerate}
\item We say that $\mathfrak{U}$ is \emph{costable} if $\mathfrak{U}\subseteq \tau\mathfrak{U}$.
\item We say that $(\mathfrak{U}_i\to \mathfrak{U})_{i\in I}$ is \emph{costable} if $\mathfrak{U}$ and $\mathfrak{U}_i$ are costable for all $i\in I$.
\item Given any costable admissible open $\mathfrak{U}$, we denote by $f\mapsto f^{(1)}$ the composition
\begin{equation}
\begin{tikzcd}[column sep=5em]
\bullet^{(1)}:\cM_{\mathfrak{X}}(\mathfrak{U}) \arrow[r,"f\mapsto f^{(1)}"] & \tau^{-1}\cM_{\mathfrak{X}}(\mathfrak{U})=\cM_{\mathfrak{X}}(\tau \mathfrak{U}) \arrow[r,"\text{restr~}\tau\mathfrak{U}\to \mathfrak{U}"] & \cM_{\mathfrak{X}}(\mathfrak{U}).
\end{tikzcd}
\end{equation}
Let $\cL\subset \cM_{\mathfrak{X}}$ be a sub-$\cO_{\mathfrak{X}}$-module sheaf. We say that $\cL$ is \emph{costable} if it verifies $\mathcal{L}(\mathfrak{U})^{(1)}\subset \mathcal{L}(\mathfrak{U})$ for any costable admissible open $\mathfrak{U}$. 
\end{enumerate}
\end{Definition}

\begin{Example}
\begin{enumerate}[label=$(\alph*)$]
    \item The admissible open $\mathfrak{D}:=\Sp \CI\langle A \rangle=\Sp \CI\langle A \rangle_1$ is costable. More generally, $\mathfrak{X}_{\rho}=\Sp \CI\langle A \rangle_\rho$ is costable if and only if $\rho\geq 1$.
    \item Both $\cO_\mathfrak{X}$ and $\cM_{\mathfrak{X}}$ are costable. The sheaf $\cO_{\mathfrak{X}}(D)$ is costable whenever $D$ is an effective rigid analytic divisor that verifies $D\geq \tau D$.
\end{enumerate}
\end{Example}

Let $\operatorname{Cat}^\tau \mathfrak{X}$ be the full subcategory of $\operatorname{Cat}\mathfrak{X}$ consisting of costable objects, and let $\operatorname{Cov}^\tau \mathfrak{X}$ be the subset of $\operatorname{Cov} \mathfrak{X}$ consisting of costable coverings. It is formal to verify that the pair $(\operatorname{Cat}^\tau \mathfrak{X},\operatorname{Cov}^\tau \mathfrak{X})$ forms a site, which we denote by $\,^\tau\underline{\mathfrak{X}}$ and refer to as the \emph{costable site of $\mathfrak{X}$}.

\section{Anderson modules over rigid analytic spaces}

\subsection*{$E(\CI)$-valued meromorphic functions}
Let $E$ be an Anderson module of dimension $d$ over $\CI$. This whole section is devoted to the construction of the sheaf of \emph{$E(\CI)$-valued functions}: it will consist in a sheaf of $A\otimes A$-modules on the costable site $\,^\tau\underline{\mathfrak{X}}$ which extends $A\hat{\otimes} E(\CI)$ on the rigid analytic unit disk $\mathfrak{D}:=\Sp \CI\langle A \rangle$. To anticipate generalizations to polar divisor and residues in the next sections, we not only focus on the holomorphic version, but also divisor versions. To materialize this, we fix an \emph{effective fractional ideal sheaf} $\mathcal{L}$ on $\mathfrak{X}$, that is, a subsheaf of $\cO_{\mathfrak{X}}$-modules $\mathcal{L}$  of $\mathcal{M}_{\mathfrak{X}}$ on $\underline{\mathfrak{X}}$ that contains $\cO_{\mathfrak{X}}$. The latter assumption ensures that, for any admissible open $\mathfrak{U}$, there is a map 
\begin{equation}\label{eq:canonical-inclusion-sheaf}
A\otimes \CI\hookrightarrow \cO_{\mathfrak{X}}(\mathfrak{U})\hookrightarrow \cL(\mathfrak{U}). 
\end{equation}
In practice, $\mathcal{L}$ will either be $\mathcal{O}_\mathfrak{X}$, $\mathcal{O}_\mathfrak{X}(D)$ for a certain effective rigid analytic divisor $D$, or $\cM_\mathfrak{X}$ itself.

\begin{Definition}
Let $\kappa:E\stackrel{\sim}{\to}\bG_a^d$ be a choice of coordinates. We define $\mathcal{E}_\kappa(\mathcal{L})(\mathfrak{U})$ and $\iota_\kappa(\mathcal{L})(\mathfrak{U})$ as resulting from the pushout square
\begin{equation}\label{eq:pushout-def}
\begin{tikzcd}
(A\otimes \CI)^d \arrow[r,"\kappa^{-1}","\sim"']\arrow[d,hookrightarrow,"\eqref{eq:canonical-inclusion-sheaf}"'] & A\otimes E(\CI) \arrow[d,"\iota_\kappa(\mathcal{L})(\mathfrak{U})"] \\
\cL(\mathfrak{U})^d \arrow[r] & \cE_{\kappa}(\mathcal{L})(\mathfrak{U})
\end{tikzcd}
\end{equation}
in the category of $A$-modules.
\end{Definition}

As the top horizontal map of diagram \eqref{eq:pushout-def} is an isomorphism, we get by the universal property of the pushout square that the lower horizontal map is an isomorphism as well. As its inverse will appear more naturally, we give it a name:
\begin{Definition}
We denote by $m_\kappa(\mathfrak{U}):\mathcal{E}_\kappa(\mathcal{L})(\mathfrak{U})\to \cL(\mathfrak{U})^d$ the inverse of the lower horizontal map in diagram \eqref{eq:pushout-def}.
\end{Definition}

Let $\mathfrak{U}\to \mathfrak{V}$ be a morphism in $\operatorname{Cat}\mathfrak{X}$ (\emph{i.e.} an inclusion of admissible opens). The restriction map $\cM_{\mathfrak{X}}(\mathfrak{V})\to \cM_{\mathfrak{X}}(\mathfrak{U})$ together with the universality of diagram \eqref{eq:pushout-def} produces a morphism
\begin{equation}\label{eq:restriction-map}
\rho_\kappa(\mathfrak{U}\to \mathfrak{V}):\cE_{\kappa}(\mathcal{L})(\mathfrak{V})\longrightarrow \cE_{\kappa}(\mathcal{L})(\mathfrak{U})
\end{equation}
of $A$-modules.  The data of $\mathcal{E}_\kappa(\mathcal{L})(\mathfrak{U})$ for all $\mathfrak{U}$ and $\rho_\kappa$ forms a contravariant functor from $\operatorname{Cat}\mathfrak{X}$ to $\mathbf{Mod}_{A}$, and hence a presheaf on the site $\underline{\mathfrak{X}}$. 

\begin{Lemma}\label{lem:Ekappa-sheaf}
$\mathcal{E}_\kappa(\mathcal{L})$ is a sheaf of $A$-modules on $\underline{\mathfrak{X}}$.
\end{Lemma}
\begin{proof}
The data of $m_\kappa=(m_\kappa(\mathfrak{U}))_{\mathfrak{U}}$ indexed over objects of $\operatorname{Cat}\mathfrak{X}$ defines an equivalence of functors among $\mathcal{E}_\kappa(\mathcal{L})$ and $\cL^d$. As $\cL^d$ is a sheaf, so is $\mathcal{E}_\kappa(\mathcal{L})$.
\end{proof}

\subsection*{Action of $A\otimes A$}\label{sec:AoA}

In this subsection we show that, under some costability assumptions, diagram \eqref{eq:pushout-def} can be upgraded to a diagram of $A\otimes A$-modules preserving the $A\otimes A$-module structure on $A\otimes E(\CI)$. Throughout this subsection, $\cL$ is a costable effective fractional ideal sheaf on $\mathfrak{X}$. \\

Let $\kappa:E\stackrel{\sim}{\to} \bG_a^d$ be a choice of coordinates. Through $\kappa$, we obtain a ring morphism $\varphi_\kappa:A\to \Mat_d(\CI\{\tau\})$ as resulting from the composition
\[
A \longrightarrow \End_{\bF\text{-vs}/\CI}(E) \stackrel{\text{via~}\kappa}{\cong} \End_{\bF\text{-vs}/\CI}(\bG_a^d) \cong \Mat_d(\CI\{\tau\})
\]
where $\CI\{\tau\}$ is the non-commutative ring of $\tau$-polynomials. We thus obtain on $(A\otimes \CI)^d$ an $A\otimes A$-module structure where, for $a\in A$, $a\otimes 1$ acts by multiplication and where $1\otimes a$ acts by $\id\otimes \varphi_\kappa(a)$. This makes the top arrow of diagram \eqref{eq:pushout-def} a morphism of $A\otimes A$-modules.\\

If $\mathfrak{U}$ is a costable admissible open, then $\mathcal{L}(\mathfrak{U})^d$ admits an action of $\varphi_\kappa$ described as follows: An element $a\in A$ acts on $F\in \mathcal{L}(\mathfrak{U})^d$ as
\begin{equation}\label{eq:action-on-L(U)}
\varphi_\kappa(a)(F)=A_0 F+A_1 F^{(1)}+...+ A_s F^{(s)} \in \mathcal{L}(\mathfrak{U})^d
\end{equation}
according to the notations of Definition \ref{def:costable}.
\begin{Remark}
If $\mathfrak{U}$ was not assumed to be costable, then identity \eqref{eq:action-on-L(U)} would only make sense in $\mathcal{L}(\mathfrak{U}\cap \tau \mathfrak{U}\cap \cdots \cap \tau^s\mathfrak{U})^d$.
\end{Remark}

This endows $\mathcal{L}(\mathfrak{U})^d$ with an $A\otimes A$-module structure. The next facts are immediate:
\begin{Lemma}
Let $\mathfrak{U}$ be a costable admissible open. Then, there exists a unique $A\otimes A$-module structure on $\mathcal{E}_\kappa(\mathcal{L})(\mathfrak{U})$ making \eqref{eq:pushout-def} a diagram of $A\otimes A$-modules with respects to the structure defined above.
If $\mathfrak{U}\to \mathfrak{V}$ is a morphism in $\operatorname{Cat}^\tau\mathfrak{X}$, the morphism \eqref{eq:restriction-map} is a morphism of $A\otimes A$-modules. 
Further, $\mathcal{E}_\kappa(\mathcal{L})$ is a sheaf of $A\otimes A$-modules on $\,^\tau\underline{\mathfrak{X}}$.
\end{Lemma}

We now wish to make the previous notion $\kappa$-free, and construct $\mathcal{E}_{\kappa}(\mathcal{L})(\mathfrak{U})$ independently of a choice of coordinates. Let $\kappa'$ be another choice of coordinates, and let 
\[
M=M_0+M_1\tau+...+M_s\tau^s \in \GL_d(\CI\{\tau\}), \quad M_i\in \GL_d(\CI),
\]
be the matrix obtained from $\kappa'\circ \kappa^{-1}$ in $\Aut_{\bF\text{-vs}/\CI}(\bG_a^d)\cong \GL_d(\CI\{\tau\})$. Let $\mathfrak{U}$ be an admissible open. By costability of $\mathcal{L}$, the morphism obtained from $M$ maps $\cL(\mathfrak{U})^d$ to $\cL(\mathfrak{U}\cap \tau \mathfrak{U}\cap \cdots \cap \tau^s\mathfrak{U})^d$. In particular, the universal property of the pushout produces a morphism
\begin{equation}\label{eq:comparison-E}
\cE_{\kappa}(\mathcal{L})(\mathfrak{U})\longrightarrow \cE_{\kappa'}(\mathcal{L})(\mathfrak{U}\cap \tau \mathfrak{U}\cap \cdots \cap \tau^s\mathfrak{U})
\end{equation}
via the commutative diagram of $A\otimes A$-modules:
\begin{equation}
\begin{tikzcd}
(A\otimes \CI)^d \arrow[rr,"\kappa^{-1}~~","\sim"']\arrow[dd,hookrightarrow]\arrow[dr,"M"] & & A\otimes E(\CI) \arrow[dd]\arrow[dr,equal] & \\
& (A\otimes \CI)^d \arrow[rr,crossing over,near start,"(\kappa')^{-1}","\sim"'] & & A\otimes E(\CI) \arrow[dd] \\
\cL(\mathfrak{U})^d \arrow[rr]\arrow[dr,"M"] & & \cE_{\kappa}(\mathcal{L})(\mathfrak{U})\arrow[dr,dashed] & \\
& \cL(\mathfrak{U}\cap \tau \mathfrak{U}\cap \cdots \cap \tau^s\mathfrak{U})^d\arrow[rr]\arrow[from=uu,hookrightarrow,crossing over] & & \cE_{\kappa'}(\mathcal{L})(\mathfrak{U}\cap \tau \mathfrak{U}\cap \cdots \cap \tau^s\mathfrak{U}) 
\end{tikzcd} \nonumber
\end{equation}
If $\mathfrak{U}$ is costable, \eqref{eq:comparison-E} becomes an isomorphism 
\begin{equation}\label{eq:comparison-E-costable}
t_{\kappa,\kappa'}:\cE_{\kappa}(\mathcal{L})(\mathfrak{U})\stackrel{\sim}{\longrightarrow} \cE_{\kappa'}(\mathcal{L})(\mathfrak{U}).
\end{equation}
\begin{Definition}
Let $\mathfrak{U}\in \operatorname{Cat}^\tau\mathfrak{X}$, $(\mathfrak{U}_i\to \mathfrak{U})_{i\in I}\in \operatorname{Cov}^\tau\mathfrak{X}$ an admissible covering.
\begin{enumerate}
\item We define $\mathcal{E}(\mathcal{L})(\mathfrak{U})$ as the $A\otimes A$-module given by the inverse limit
\[
\mathcal{E}(\mathcal{L})(\mathfrak{U}) =\varprojlim \cE_{\kappa}(\mathcal{L})(\mathfrak{U})
\]
taken over the projective system of choices of coordinates $\kappa$, the transition maps being $t_{\kappa,\kappa'}$ as in \eqref{eq:comparison-E-costable}.
\item We denote by $\varepsilon_{\kappa}(\mathcal{L})(\mathfrak{U}):\mathcal{E}(\mathcal{L})(\mathfrak{U}) \to \mathcal{E}_\kappa(\mathcal{L})(\mathfrak{U})$ the isomorphism given by universality of the inverse limit.
\item We let $\iota(\mathfrak{U}):A\otimes E(\CI)\to \mathcal{E}(\mathcal{L})(\mathfrak{U})$ be the composition 
\[
A\otimes E(\CI) \stackrel{\iota_{\kappa}(\mathfrak{U})}{\longrightarrow} \mathcal{E}_{\kappa}(\mathcal{L})(\mathfrak{U}) \stackrel{\varepsilon_\kappa(\mathfrak{U})^{-1}}{\longrightarrow} \mathcal{E}(\mathcal{L})(\mathfrak{U}).
\]
\end{enumerate}
\end{Definition}
\begin{Remark}
Note that $\iota$ is independent of the choice of $\kappa$, as for another choice of $\kappa'$, the diagram of $A\otimes A$-modules
\begin{equation}
\begin{tikzcd}[column sep=3em]
A\otimes E(\CI) \arrow[r,"\iota_\kappa(\mathfrak{U})"]\arrow[dr,"\iota_{\kappa'}(\mathfrak{U})"'] & \mathcal{E}_{\kappa}(\mathcal{L})(\mathfrak{U})\arrow[d,"t_{\kappa,\kappa'}"] & \mathcal{E}(\mathcal{L})(\mathfrak{U}) \arrow[l,"\varepsilon_\kappa(\mathfrak{U})"']\arrow[dl,"\varepsilon_{\kappa'}(\mathfrak{U})"] \\
& \mathcal{E}_{\kappa'}(\mathcal{L})(\mathfrak{U}) & 
\end{tikzcd}
\nonumber
\end{equation}
commutes.
\end{Remark}

Proposition \ref{prop:intro-E-valuedfunctions} of the Introduction is then an easy consequence of the definitions and the following fact:
\begin{Proposition}\label{prop:E-sheaf}
$\mathcal{E}(\mathcal{L})$ is a sheaf of $A\otimes A$-modules on $\,^\tau\underline{\mathfrak{X}}$.
\end{Proposition}
\begin{proof}
Let $(\mathfrak{U}_i\to\mathfrak{U})_{i\in I}\in \operatorname{Cov}^{\tau}\mathfrak{X}$. We need to verify that the sequence
\[
\mathcal{E}(\mathcal{L})(\mathfrak{U})\to \prod_{i\in I}{\mathcal{E}(\mathcal{L})(\mathfrak{U}_i)} \rightrightarrows \prod_{i,j\in I}{\mathcal{E}(\mathcal{L})(\mathfrak{U}_i\times_{\mathfrak{U}}\mathfrak{U}_j)}
\]
is exact. But, given $\kappa$, the above is canonically isomorphic to
\[
\mathcal{E}_\kappa(\mathcal{L})(\mathfrak{U})\to \prod_{i\in I}{\mathcal{E}_\kappa(\mathcal{L})(\mathfrak{U}_i)} \rightrightarrows \prod_{i,j\in I}{\mathcal{E}_\kappa(\mathcal{L})(\mathfrak{U}_i\times_{\mathfrak{U}}\mathfrak{U}_j)}
\]
which is exact by Lemma \ref{lem:Ekappa-sheaf}.
\end{proof}

As one verifies, this definition is consistent with the definition of $A\hat{\otimes}E(\CI)$ present in \cite{gazda-maurischat}. Recall that $\mathfrak{D}=\Sp \CI \langle A \rangle$. The relevant statement is the following:
\begin{Lemma}\label{lem:E-on-disk}
The admissible open $\mathfrak{D}\subset \mathfrak{X}$ is costable. Further, the morphism 
\[
\iota(\mathfrak{D}):A\otimes E(\CI)\to \mathcal{E}(\cL)(\mathfrak{D})
\]
factors uniquely through the completion map $A\otimes E(\CI)\to A\hat{\otimes} E(\CI)$. The image of the latter is $\mathcal{E}(\cO_{\mathfrak{X}})(\mathfrak{D})$. 
\end{Lemma}
\begin{proof}
That $\Sp \CI\langle A \rangle$ is costable only means that $\tau \fm$ is again maximal in $\CI\langle A \rangle$ for any maximal ideal $\fm\subset \CI\langle A \rangle$. Now, recall that the choice of $\kappa$ gives rise to a metric on $A\otimes E(\CI)$ which is independent of $\kappa$, and that $A\hat{\otimes} E(\CI)$ is the completion associated to the latter metric. We obtain by definition a pushout square of $A\otimes A$-modules:
\begin{equation}
    \begin{tikzcd}
    (A\otimes \CI)^d \arrow[r,"\kappa^{-1}","\sim"']\arrow[d,hookrightarrow] & A\otimes E(\CI) \arrow[d] \\
    \CI\langle A \rangle^d \arrow[r,"\sim"] & A\hat{\otimes} E(\CI)
    \end{tikzcd}
\end{equation}
As $\cO_{\mathfrak{X}}(\mathfrak{D})=\CI\langle A \rangle$, we obtain a unique map $c_\kappa:A\hat{\otimes} E(\CI)\to \mathcal{E}_\kappa(\mathcal{L})(\mathfrak{D})$, and this map is an isomorphism. Now, one verifies that
\[
A\hat{\otimes} E(\CI) \xrightarrow{(c_\kappa)_\kappa}\varprojlim_{\kappa} \mathcal{E}_\kappa(\cO_\mathfrak{X})(\mathfrak{D})=\mathcal{E}(\mathcal{O}_\mathfrak{X})(\mathfrak{D})
\]
is a well-defined isomorphism of $A\otimes A$-modules. 
\end{proof} 

\subsection*{Special functions as $E(\CI)$-valued functions}
We are now in position to introduce the object of study. 

In \cite[Def. 3.1]{gazda-maurischat}, we defined the $A$-module of special functions of $E$ as
\[ \fsf(E)= \left\{ \omega \in A\hat{\otimes} E(\CI) \,\,\middle|\,\,
\forall a\in A: (a\otimes 1)\omega = (1\otimes a)\omega \right\}. \]
According to Lemma \ref{lem:E-on-disk}, the $A$-module $\fsf(E)$ admits the following generalization:
\begin{Definition}\label{def:sf}
Let $\mathcal{L}$ be a costable effective fractional ideal sheaf on $\fX$, and
let $\mathfrak{U}$ be an object of $\operatorname{Cat}^{\tau}\mathfrak{X}$. The \emph{module of special functions of $\mathcal{E}(\mathcal{L})$ on $\mathfrak{U}$} is the $A$-module
\[
\fsf_E(\mathcal{L})(\mathfrak{U})=\left\{\omega \in \mathcal{E}(\mathcal{L})(\mathfrak{U})\,\,\middle|\,\, \forall a\in A:~(a\otimes 1) \omega=(1\otimes a)\omega \right\}.
\]
\end{Definition}

In particular as $\mathcal{E}(\mathcal{O}_\mathfrak{X})(\mathfrak{D})=A\hat{\otimes}E(\CI)$, we have  $\fsf(E)=\fsf_E(\cO_\mathfrak{X})(\mathfrak{D})$ as $A$-modules.\\ We have the following:
\begin{Proposition}\label{prop:sf-sheaf}
$\fsf_E(\mathcal{L})$  is a sheaf of $A$-modules on $\,^\tau\underline{\mathfrak{X}}$.
\end{Proposition}
\begin{proof}
This follows immediately from Definition \ref{def:sf} and the fact that $\cE$ is a sheaf on $\,^\tau\underline{\mathfrak{X}}$.
\end{proof}

\section{Continuation and residues of special functions}\label{sec:continuation-residue}
\subsection*{Residue at $\fj$}
\subsubsection*{Residue map for meromorphic functions on $\mathfrak{X}$}
We begin with a brief recall of residues for meromorphic functions on $\mathfrak{X}$. Recall that we have a morphism of locally ringed spaces $f:\mathfrak{X}\to X$ given by the GAGA functor, and that any sheaf $\mathcal{F}$ on $X$ induces a sheaf $\cF^{\rig}=f^*\cF$ on $\mathfrak{X}$. The \emph{sheaf of holomorphic differentials}, denoted $\Omega^{\operatorname{hol}}_{\mathfrak{X}/\CI}$, is defined as 
\[
\Omega^{\operatorname{hol}}_{\mathfrak{X}/\CI}:=\left(\Omega_{X/\CI}^1\right)^{\rig}=\mathcal{O}_\mathfrak{X}\otimes_{\cO_X}\Omega^1_{\cO_X/\CI}.
\]
\begin{Remark}
In \cite[\S 3.6]{FVDP}, $\Omega^{\operatorname{hol}}_{\mathfrak{X}/\CI}$ is rather called the \emph{universal finite differential module} and denoted by  $\Omega^{\operatorname{f}}_{\mathfrak{X}/\CI}$.
\end{Remark}

Similarly, one defines the \emph{sheaf of meromorphic differentials}, denoted $\Omega^{\operatorname{mer}}_{\mathfrak{X}/\CI}$,  as 
\[
\Omega^{\operatorname{mer}}_{\mathfrak{X}/\CI}:=\cM_\mathfrak{X}\otimes_{\cO_\mathfrak{X}}\Omega^{\operatorname{hol}}_{\mathfrak{X}/\CI}=\mathcal{M}_\mathfrak{X}\otimes_{\cO_X}\Omega^1_{\cO_X/\CI}.
\]
Identifying $\Omega^{\operatorname{mer}}_{\mathfrak{X}/\CI}$ with $\cM_\mathfrak{X}\otimes_A \Omega_{A/\bF}^1$, we recognize that there is a map
\[
\tau:\Omega^{\operatorname{mer}}_{\mathfrak{X}/\CI}\longrightarrow \tau^{-1}\Omega^{\operatorname{mer}}_{\mathfrak{X}/\CI}
\]
given as $\tau\otimes_A \id_{\Omega_{A}}$ on $\cM_\mathfrak{X}\otimes_A \Omega_{A/\bF}^1$. Therefore, if $\mathfrak{U}$ is a costable admissible open of $\mathfrak{X}$, we obtain a morphism $\tau:\Omega^{\operatorname{mer}}_{\mathfrak{X}/\CI}(\mathfrak{U})\to \Omega^{\operatorname{mer}}_{\mathfrak{X}/\CI}(\mathfrak{U})$. \\

Given $x\in |\mathfrak{X}|=|X|$, we hence have a residue map
\[
\operatorname{residue}_x:(\Omega^{\operatorname{mer}}_{\mathfrak{X}/\CI})_x\cong (\Omega_{X/\CI}^1)_x\longrightarrow \CI, \quad \left(\sum_{i\geq -h}{a_i \pi_x^i}\right)d\pi_x\longmapsto a_{-1}
\]
obtained from the classical algebraic residue map, and independent on the choice of a uniformizer $\pi_x\in \cO_{\mathfrak{X},x}$.

\subsubsection*{Residue of $E(\CI)$-valued functions at $\fj$}
Now that we have constructed a sheaf of $E(\CI)$-valued meromorphic functions, let us discuss in which sense one might take the residue at $\fj\in \mathfrak{X}$. We denote by $\mathcal{J}$ the subsheaf of $\cM_\mathfrak{X}$ of functions that are meromorphic on $\mathfrak{X}$, yet holomorphic away from the support of $J$; namely, 
\[
\mathcal{J}:=\operatorname{colim}_n \cO_\mathfrak{X}(n\cdot J).
\]
As all sheaves $\cO_\mathfrak{X}(n\cdot J)$ are costable, so is $\cJ$.
\begin{Remark}
Note that $\mathcal{J}(\mathfrak{U})$ does generally not correspond to $\bigcup_n \mathcal{O}_{\mathfrak{X}}(n\cdot J)(\mathfrak{U})$ for an arbitrary admissible open $\mathfrak{U}$: the meromorphic functions with only poles the points $\tau^n\mathfrak{j}$ with order $n$, for all positive $n$, will typically belong to $\mathcal{J}(\mathfrak{X})$ but not to $\bigcup_n \mathcal{O}_{\mathfrak{X}}(n\cdot J)(\mathfrak{X})$.
\end{Remark}

Let $\kappa:E\stackrel{\sim}{\to}\bG_a^d$ be a choice of coordinates for $E$. It induces a map of $\CI$-vector spaces 
\[
\Lie_\kappa:=\Lie_{\kappa}(\CI):\Lie_E(\CI)\longrightarrow \Lie_{\bG_a^d}(\CI)=\CI^d.
\]
\begin{Definition}
We define $\res_\fj$ as the unique map of $A$-modules making the following diagram commute:
\begin{equation}
    \begin{tikzcd}
    \Omega^1_{A/\bF}\otimes_A\mathcal{E}(\mathcal{J})(\mathfrak{X})\arrow[d,"m_\kappa\circ \varepsilon_\kappa"',"\wr"]\arrow[r,dashed,"\res_\fj"] & \Lie_E(\CI) \arrow[d,"\wr"',"\Lie_\kappa"] \\
    \Omega_{A/\bF}^1\otimes_A\mathcal{J}(\mathfrak{X})^d \arrow[r,"\operatorname{residue}_{\fj}"] & \CI^d
    \end{tikzcd}
\end{equation}
\end{Definition}

As suggested by the notation, the next proposition proves the independence of $\res_\fj$ on the choice of coordinates.
\begin{Proposition}\label{prop:res-indep}
The map $\res_\fj$ is independent of $\kappa$. 
\end{Proposition}
\begin{proof}
First observe that if $f\in  \mathcal{J}(\mathfrak{X})$ then $f^{(1)}$ belongs to $\mathcal{J}(\mathfrak{X})^{(1)}$ which consists of functions that are meromorphic on $\mathfrak{X}$ and holomorphic on the admissible open $\mathfrak{X}\setminus\{\tau\fj,\tau^2\fj,...\}$, and $f^{(1)}$ is thus regular at $\fj$. In particular, given $F\in (\Omega_{A/\bF}^1\otimes_A\mathcal{J}(\mathfrak{X}))^d$ and a matrix $M=M_0+M_1\tau+...\in \GL_d(\CI\{\tau\})$ where $M_i\in \Mat_d(\CI)$ and $M_0\in \GL_d(\CI)$, we have 
\begin{align*}
\operatorname{residue}_{\fj}\left(M\cdot F\right) &=  \operatorname{residue}_{\fj}\left(M_0\cdot F+\underbrace{M_1\cdot F^{(1)}+M_2\cdot F^{(2)}+...}_{\text{regular~at~}\fj}\right)= \operatorname{residue}_{\fj}(M_0\cdot F) \\
 &=M_0\cdot \operatorname{residue}_{\fj}(F).
\end{align*}
Now, for another choice $\kappa'$ of coordinates for $E$, $\kappa'$ differs from $\kappa$ by an element in $\Aut_{\bF_q-\text{vs}/\CI}(\bG_a^d)$, and hence is represented by a matrix $M\in \GL_d(\CI\{\tau\})$. We hence have a diagram:
\begin{equation}
    \begin{tikzcd}
    & \Omega^1_{A/\bF}\otimes_A\mathcal{E}(\mathcal{J})(\mathfrak{X})\arrow[d,"m_\kappa\circ \varepsilon_\kappa"',"\wr"]\arrow[r,"\res_\fj"]\arrow[ddl,bend right,"m_{\kappa'}\circ \varepsilon_{\kappa'}"'] & \Lie_E(\CI) \arrow[d,"\wr"',"\Lie_\kappa"] \arrow[ddr,bend left,"\Lie _{\kappa'}"] &  \\
    & \Omega_{A/\bF}^1\otimes_A\mathcal{J}(\mathfrak{X})^d \arrow[r,"\operatorname{residue}_{\fj}"] \arrow[dl,"M"'] & \CI^d \arrow[dr,"M_0"] & \\
    \Omega_{A/\bF}^1\otimes_A\mathcal{J}(\mathfrak{X})^d \arrow[rrr,"\operatorname{residue}_{\fj}"]& & & \CI^d
    \end{tikzcd}\nonumber
\end{equation}
where all inner squares commute. In particular, the outer square commutes, which proves the claim.
\end{proof}

\subsection*{Proofs of the main results}
The next result, already announced in the introduction, translates the fact that special functions admit a unique meromorphic continuation to $\mathfrak{X}$, with poles at $\fj$ and its iterates of order at most $d=\dim E$. For a rigid analytic divisor $D$, we rather use the notation $\fsf_E(D)$ for $\fsf_E(\mathcal{O}_\mathfrak{X}(D))$.
\begin{Theorem}[Meromorphic continuation of special functions]\label{thm:continuation-sf}
The canonical $A$-module map $\fsf_E(d\cdot J)(\mathfrak{X})\to \fsf_E(d\cdot J)(\mathfrak{D})=\fsf_E(\cO_\mathfrak{X})(\mathfrak{D})=\fsf(E)$ is an isomorphism.
\end{Theorem}

We begin with some preliminary results. Recall that there exists a unique $A$-linear morphism
\[
\exp_E:\Lie_E(\CI) \longrightarrow E(\CI)
\]
which is \emph{analytic} in the following sense: for any choice of $\kappa:E\stackrel{\sim}{\to}\bG_a^d$, the map $\exp_E^\kappa:\CI^d\to \CI^d$, obtained as $\kappa\circ \exp_E \circ (\Lie_\kappa)^{-1}$, can be written in the form
\begin{equation}
\exp^{\kappa}_E(x)=e_0^{\kappa}x+e_1^{\kappa}\tau(x)+e_2^{\kappa}\tau^2(x)+... \quad \text{where}~e_i^{\kappa}\in \Mat_{d}(\CI), \nonumber
\end{equation}
the series converging for any $x\in \CI^d$ \cite[\S2]{anderson}. By continuity, $\exp_E$ extends $A$-linearly to 
\[
\widehat{\exp}_E:A\hat{\otimes}\Lie_E(\CI)\longrightarrow A\hat{\otimes} E(\CI)
\]
(see \cite[\S 2]{gazda-maurischat}). \\

Let $u\in A$ be a separating element. We recall from \cite[Thm. 3.11]{gazda-maurischat} that we have an isomorphism of $A$-modules 
\[
\delta_u:\fd_{A/\bF[u]}\otimes_A \Lambda_E\stackrel{\sim}{\longrightarrow} \fsf(E)
\]
mapping an element $\lambda\in \fd_{A/\bF[u]}\otimes_A \Lambda_E$ to $\omega_\lambda:=\widehat{\exp}_E((u\otimes 1-1\otimes \partial \varphi_{u})^{-1}(\lambda))$ seen in $A\hat{\otimes} E(\CI)$.

\begin{Remark}
The map $\delta_u$ is not exactly the one that appears in \cite{gazda-maurischat}, but minus that one.
\end{Remark}

\begin{Lemma}\label{lem:explicit-formula-for-sf}
Let $n_u$ be the $\CI$-linear endomorphism of $\Lie_E(\CI)$ given by $\partial\varphi_u-\ell(u)$. Then, the composition
\[
\fsf(E)\hookrightarrow \mathcal{E}(\mathcal{O}_\mathfrak{X})(\mathfrak{D})\stackrel{\sim}{\longrightarrow} \cO_{\mathfrak{X}}(\mathfrak{D})^d=\CI\langle A \rangle^d,
\]
where the second arrow is $(m_\kappa\circ \varepsilon_\kappa)(\mathfrak{D})$, maps $\omega_\lambda$ to the converging sum
\[
\sum_{j=0}^{d-1}{\sum_{i=0}^{\infty}{\frac{(1\otimes e_i^\kappa)[(1\otimes \Lie_\kappa)(1\otimes n_u^j)(\lambda)]^{(i)}}{(u\otimes 1-1\otimes \ell(u)^{q^i})^{j+1}} }}.
\]
\end{Lemma}
\begin{proof}
Let us first contemplate the diagram
\begin{equation}
\begin{tikzcd}
A\hat{\otimes} \Lie_E(\CI) \arrow[r,"\widehat{\exp}_E"]\arrow[d,"1\otimes \Lie_E \kappa"']& A\hat{\otimes}E(\CI) \arrow[r,"\sim"]\arrow[d,"1\otimes \kappa"] & \mathcal{E}(\mathcal{O}_\mathfrak{X})(\mathfrak{D})\arrow[d,"(m_\kappa\circ \varepsilon_\kappa)(\mathfrak{D})"] \\
\CI\langle A \rangle^d \arrow[r,"\widehat{\exp}_E^{\kappa}"] & \CI\langle A \rangle^d \arrow[r,"="] & \cO_{\mathfrak{X}}(\mathfrak{D})^d
\end{tikzcd}
\end{equation}
where we defined $\widehat{\exp}^{\kappa}_E$ as the unique map making the left-hand square commute. That the right-hand square commutes is a matter of definition. 

Now, observe that $n_u$ is nilpotent of order at most $d$ by definition of an Anderson $A$-module. Hence, we have
\begin{equation}
(u\otimes 1-1\otimes \partial \varphi_{u})^{-1}=\sum_{j=0}^{d-1}{\frac{1\otimes n_u^j}{(u\otimes 1-1\otimes \ell(u))^{j+1}}}. \nonumber
\end{equation}
and then, seeing $\omega_\lambda\in A\hat{\otimes} E(\CI)$,
\begin{align}\label{computation-residue}
(1\otimes \kappa)(\omega_\lambda)&=((1\otimes\kappa)\circ \widehat{\exp}_E)\left((u\otimes 1-1\otimes \partial \varphi_{u})^{-1}(\lambda)\right) \nonumber\\
&= [\widehat{\exp}_E^{\kappa}\circ (1\otimes \Lie_\kappa)]\left( \sum_{j=0}^{d-1}{\frac{1\otimes n_u^j}{(u\otimes 1-1\otimes \ell(u))^{j+1}}}\right)(\lambda) \nonumber \\
&=\sum_{j=0}^{d-1}{\sum_{i=0}^{\infty}{\frac{(1\otimes e_i^\kappa)[(1\otimes \Lie_\kappa)(1\otimes n_u^j)(\lambda)]^{(i)}}{(u\otimes 1-1\otimes \ell(u)^{q^i})^{j+1}} }}.
\end{align}
That the upper series converge is immediate from the convergence of $\exp^\kappa_E$.
\end{proof}

\begin{proof}[Proof of Theorem \ref{thm:continuation-sf}]
Let $u\in A$ be a separating element. We have a diagram of $A$-modules
\begin{equation}
\begin{tikzcd}[column sep=4em, row sep=4em]
\fsf_E(d\cdot J_u)(\mathfrak{X}) \arrow[r,hook]\arrow[d,hook] & \mathcal{E}(d\cdot J_u)(\mathfrak{X}) \arrow[r,"(m_\kappa\circ \varepsilon_\kappa)(\mathfrak{X})","\sim"']\arrow[d,hook] & \mathcal{O}_\mathfrak{X}(d\cdot J_u)(\mathfrak{X})^d \arrow[d,hook]  \\
\fsf_E(\cO_\mathfrak{X})(\mathfrak{D}) \arrow[r,hook]\arrow[rru,dashed]\arrow[ru,dashdotted] & \mathcal{E}(\mathcal{O}_\mathfrak{X})(\mathfrak{D}) \arrow[r,"(m_\kappa\circ \varepsilon_\kappa)(\mathfrak{D})","\sim"'] & \mathcal{O}_\mathfrak{X}(\mathfrak{D})^d=\CI\langle A \rangle^d 
\end{tikzcd}
\end{equation}
where, by the previous Lemma \ref{lem:explicit-formula-for-sf}, the composition of the lower horizontal maps factors through the right-hand vertical map (represented as the dashed arrow). The map $(m_\kappa\circ \varepsilon_\kappa)(\mathfrak{X})$ being an isomorphism, the latter factors through $\mathcal{E}(d\cdot J_u)(\mathfrak{X})$ (represented as the dashdotted arrow). As by construction, $\fsf_E(d\cdot J_u)(\mathfrak{X})$ is the intersection of $\fsf_E(\cO_\mathfrak{X})(\mathfrak{D})$ and $\mathcal{E}(d\cdot J_u)(\mathfrak{X})$ in $\mathcal{E}(\mathcal{O}_\mathfrak{X})(\mathfrak{D})$, we thus obtain that the inclusion
\[
\fsf_E(d\cdot J_u)(\mathfrak{X}) \hookrightarrow \fsf_E(\cO_\mathfrak{X})(\mathfrak{D})
\]
is an equality. As this holds for all $u\in A$ separating, Proposition \ref{prop:O(J)-equals-intersection-O(Ju)} implies 
\[
\fsf_E(d\cdot J)=\bigcap_{\substack{u\in A \\ u~\text{separating}}} \fsf_E(d\cdot J_u)
\]
hence that $\fsf_E(d\cdot J)(\mathfrak{X}) =\fsf_E(\cO_\mathfrak{X})(\mathfrak{D})$ as desired.
\end{proof}

Let $u$ be a separating element in $A$. The relative different ideal $\fd_{A/\bF[u]}$ of the separable extension $K/\bF(u)$ is generated by the elements of the form $f_s'(s)$ where $s\in A$ is such that $K=\bF(s,u)$ and $f_s$ is the minimal polynomial of $s$ over $\bF(u)$. It is also described as 
\begin{equation}
\fd_{A/\bF[u]}=\{x\in A~|~\forall y\in A:~xdy\in Adu\}. \nonumber
\end{equation}
As such, we have an $A$-linear isomorphism $\mu:\Omega_{A/\bF}^1\otimes_A \fd_{A/\bF[u]}\stackrel{\sim}{\to} Adu$ given by multiplication $\omega\otimes x\mapsto x\omega$. The main result of this note is the following:

\begin{Theorem}\label{thm:res-inverse}
We have $\res_\fj(\Omega^1_{A/\bF}\otimes_A \fsf_E(d\cdot J))= \Lambda_E$ as submodules of $\Lie_E(\CI)$. In addition, the diagram
\begin{equation}\label{eq:residue-diagram}
\begin{tikzcd}
\Omega_{A/\bF}^1\otimes_A \fd_{A/\bF[u]}\otimes_A \Lambda_E \arrow[d,"\mu\otimes \id_{\Lambda}"',"\wr"]\arrow[r,"\id_{\Omega}\otimes_A \delta_u","\sim"'] & \Omega_{A/\bF}^1\otimes_A\fsf(E)  \\
\Lambda_E & \Omega_{A/\bF}^1\otimes_A\fsf_E(d\cdot J)(\mathfrak{X}) \arrow[l,"\res_\fj"'] \arrow[u,"\operatorname{Thm}.~\ref{thm:continuation-sf}"',"\wr"]
\end{tikzcd}
\end{equation}
commutes.
\end{Theorem}
\begin{proof}
The differential $du$ has a unique preimage $\mu^{-1}(du)\in\Omega_{A/\bF}^1\otimes_A \fd_{A/\bF[u]}$ which we write as 
\begin{equation}
Du:=\sum_i{dt_i\otimes_A f_{s_i}'(s_i)}. \nonumber
\end{equation}
The inverse $\Lambda_E\stackrel{\sim}{\to} \Omega_{A/\bF}^1\otimes_A \fd_{A/\bF[u]}\otimes_A \Lambda_E$ of the left-vertical map in diagram \eqref{eq:residue-diagram} maps $\lambda$ to $Du\otimes_A \lambda=\sum_i{dt_i\otimes_A f_{s_i}'(s_i)\otimes_A \lambda}$. Once composed with $\id_\Omega\otimes_{A}\delta_u$, it is mapped to 
\begin{equation}
\sum_i{dt_i\otimes_A \widehat{\exp}_E\left((u\otimes 1-1\otimes \partial \varphi_u)^{-1}(f'_{s_i}(s_i)\otimes 1)(\lambda)\right)}\in \Omega_{A/\bF}^1\otimes_A \fsf_E(d\cdot J)(\mathfrak{X}). \nonumber
\end{equation}
To conclude, we need to compute $\res_\fj$ of the above. Through a choice of coordinates $\kappa:E\stackrel{\sim}{\to}\bG_a^d$, it boils down thanks to Lemma \ref{lem:explicit-formula-for-sf} to the computation of 
\begin{equation}
\sum_i{\operatorname{residue}_{\fj}\left(\sum_{j=0}^{d-1}{\sum_{n=0}^{\infty}{dt_i\otimes_A\frac{ f'_{s_i}(s_i)\otimes \Lie_\kappa^{-1}\left( e_n^\kappa(\Lie_\kappa(n_u^j(\lambda)))^{(n)}\right)}{(u\otimes 1-1\otimes \ell(u)^{q^n})^{j+1}} }}\right)}. \nonumber
\end{equation}
Only the terms in $(j,n)=(0,0)$ are relevant, and we are left with:
\begin{equation}
\sum_i{\operatorname{residue}_{\fj}\left(dt_i\otimes_A \frac{f'_{s_i}(s_i)\otimes \lambda}{u\otimes 1-1\otimes \ell(u)}\right)} = \operatorname{residue}_{\fj}\left(\frac{du \otimes \lambda}{u\otimes 1-1\otimes \ell(u)}\right) = \lambda,
\end{equation}
since $u\otimes 1-1\otimes \ell(u)$ is a uniformizer at $\fj$.
\end{proof}

\begin{Corollary}\label{cor:no-pole-no-sf}
We have $\fsf_E(\mathcal{O}_\mathfrak{X})(\mathfrak{X})=(0)$.
\end{Corollary}
\begin{proof}
Since $\Omega_{A/\bF}^1$ is invertible as an $A$-module, it suffices to show that $\Omega_{A/\bF}^1\otimes_A \fsf_{E}(\mathcal{O}_\mathfrak{X})(\mathfrak{X})$ is trivial. Yet, any element $\eta$ in the latter module satisfies $\res_\fj(\eta)=0$. This implies $\eta=0$ by Theorem \ref{thm:res-inverse}.
\end{proof}

\subsection*{Filtration on $\fsf(E)$}\label{subsec:filtration}
As already mentioned in the introduction, the $A$-module $\fsf(E)$ carries a natural increasing finite filtration indexed by the order of poles along the divisor $J$:
\begin{equation}\label{eq:filtration-fsf}
(0)\stackrel{\text{Cor~}\ref{cor:no-pole-no-sf}}{=} \fsf_E(\cO_\mathfrak{X})(\mathfrak{X}) \hookrightarrow \fsf_E(J)(\mathfrak{X})\hookrightarrow \cdots \hookrightarrow \fsf_E(d\cdot J)(\mathfrak{X})\stackrel{\text{Thm~}\ref{thm:continuation-sf}}{=} \fsf(E).
\end{equation}
We conclude this text by giving a few examples of the above filtration built out from the Carlitz $t$-module. These examples show that the filtration steps can be quite arbitrary.

\subsubsection*{Carlitz tensor powers}
Let $C=\bP^1_\bF$ and let $\infty$ be the point $[0:1]$. We identify $A$ with $\bF[t]$ and $A\otimes \CI$ with the polynomial ring $\CI[t]$. We also denote $\ell(t)$ by $\theta$. Let $n>0$. The \emph{$n$th tensor power of the Carlitz $t$-module over $\CI$} is the unique $A$-module $\mathsf{C}^{\otimes n}$ which equals $\bG_a^n$ as an $\bF$-vector space scheme, and for which 
\[
\varphi(t)=\begin{pmatrix}
\theta & 1 & & \\
 & \theta & \ddots & \\
 & & \ddots & 1 \\
\tau & & & \theta  
\end{pmatrix} \in \Mat_n(\CI\{\tau\})\cong \End_{\bF-\text{vs}/\CI}(\bG_a^n).
\]
As sub-$\bF[t]$-modules of $A\hat{\otimes}\mathsf{C}^{\otimes n}(\CI)=\CI\langle t \rangle^n$, we have 
\[
\fsf(\mathsf{C}^{\otimes n})=\bF[t]\cdot \begin{pmatrix}
1 \\ (t-\theta) \\ \vdots \\ (t-\theta)^{n-1}
\end{pmatrix}\omega(t)^n,
\]
where $\omega$ is the Anderson-Thakur function as in \eqref{eq:omega}. Since $\omega(t)^n$ has a pole of order $n$ at $\fj=(t-\theta)$, the filtration \eqref{eq:filtration-fsf} has only one jump at the $n$th degree:
\[
(0)=\fsf_{\mathsf{C}^{\otimes n}}((n-1)\cdot J)(\mathfrak{X})\subsetneq \fsf_{\mathsf{C}^{\otimes n}}(n\cdot J)(\mathfrak{X})=\fsf(\mathsf{C}^{\otimes n}).
\]

\subsubsection*{Sum of Carlitz tensor powers}
Let $E$ be the product of $m$ copies of $\mathsf{C}^{\otimes n}$. Then, $E$ has dimension $mn$, and $\fsf(E)=\fsf(\mathsf{C}^{\otimes n})^{\oplus m}$. In particular, 
the filtration \eqref{eq:filtration-fsf} has only one jump at the $n$th degree:
\[
(0)=\fsf_{E}((n-1)\cdot J)(\mathfrak{X})\subsetneq \fsf_{E}(n\cdot J)(\mathfrak{X})=\fsf_{E}(mn\cdot J)(\mathfrak{X})=\fsf(E).
\]

\subsubsection*{Prolongations of the Carlitz module}
For $k\geq 0$, let $\rho_k\mathsf{C}$ be the $k$th prolongation of $\mathsf{C}$ as defined in \cite{maurischat-prolongations}. It is an Anderson $A$-module which equals $\bG_a^{k+1}$ as an $\bF$-module scheme. The module $\fsf(\rho_k\mathsf{C})$ is free rank $k+1$ over $\bF[t]$, having the elements 
\[
\omega_1:=\begin{pmatrix}
\omega(t) \\ 0 \\ \vdots \\ 0
\end{pmatrix}, \quad 
\omega_2:=\begin{pmatrix}
\partial_t^{(1)}\omega(t) \\ \omega(t) \\ \vdots \\ 0
\end{pmatrix},~\cdots~,
\omega_{k+1}:=\begin{pmatrix}
\partial_t^{(k)}\omega(t) \\ \partial_t^{(k-1)}\omega(t) \\ \vdots \\ \omega(t)
\end{pmatrix}.
\]
for basis (here $\partial_t^{(j)}$ refers to the $j$th higher derivative). Therefore, the filtration \eqref{eq:filtration-fsf} has exactly $(k+1)$-jumps, and 
\[
\fsf_{\rho_k\mathsf{C}}(d\cdot  J)(\mathfrak{X})=\bigoplus_{i=1}^{d}\bF[t]\cdot \omega_i.
\]

\subsubsection*{Summary}

For appropriate choices of $n$, $m$, and $k$
in these three cases, we obtain $t$-modules of dimension $d$ with different filtration steps.
\begin{itemize}
    \item For the $d$th Carlitz tensor power $\mathsf{C}^{\otimes d}$, there is one step and this one is at the $d$th degree.
\item For the product of $d$ copies of $\mathsf{C}$, \emph{i.e.}~$\mathsf{C}^{\oplus d}$, there is also one step, but this time at degree $1$.
\item For the $(d-1)$st prolongation, there is a step at each degree.
\end{itemize}
Of course, these examples can easily be extended, \emph{e.g.} to products of different tensor powers of the Carlitz module, to obtain even more variation in the steps.
But this is not in the scope of this manuscript.

\def\cprime{$'$}


\begin{thebibliography}{Mau18b}

\bibitem[And86]{anderson}
G. Anderson. \emph{$t$-motives}, Duke Math. J., 53(2):457--502, 1986.

\bibitem[AP15]{angles}
B. Angl{\`e}s, F. Pellarin. \emph{Universal Gauss-Thakur sums and $L$-series}, Invent. Math., 200(2):653--669, 2015.

\bibitem[AT90]{andersonthakur} G. Anderson, D. Thakur, \emph{Tensor powers of the Carlitz module and zeta values}, Ann. of Math. (2), 132(1):159--191, 1990.

\bibitem[Bos82]{bosch-meromorph}
S. Bosch. \emph{Meromorphic functions on proper rigid analytic varieties}, Séminaire de Théorie des Nombres de Bordeaux, pages 1--22, 1982.

\bibitem[Bos14]{bosch}
S. Bosch, \emph{Lectures on formal and rigid geometry}, Lecture Notes in Mathematics 2105, Springer, 2014.

\bibitem[Fer22]{ferraro22} G.~H. Ferraro,
\emph{A class of functional identities associated to curves over finite fields}, Preprint
\href{https://arxiv.org/abs/2212.07823}{arXiv:2212.07823}, 2022.

\bibitem[Fer23]{ferraro23}
G.~H. Ferraro, \emph{Special functions and dual special functions in drinfeld modules of arbitrary rank}, Preprint \href{https://arxiv.org/abs/2303.11468}{arXiv:2303.11468}, 2023.

\bibitem[FvdP04]{FVDP}
J. Fresnel, M. van~der Put, \emph{Rigid analytic geometry and its applications}, Progress in Mathematics vol. 218, Birkh\"{a}user Boston, 2004.

\bibitem[GJ23]{gazdajunger} Q. Gazda, D. Junger, \emph{Pour une définition commune des courbes elliptiques et modules de Drinfeld}, Preprint \href{https://arxiv.org/abs/2306.13160}{arXiv:2306.13160}, 2023.

\bibitem[GM21]{gazda-maurischat}
Q. Gazda, A. Maurischat, \emph{Special functions and gauss–thakur sums in higher rank and dimension}, Journal für die reine und angewandte Mathematik (773):231--261, 2021.

\bibitem[GP18]{green-papanikolas}
N. Green, M. Papanikolas, \emph{Special $L$-values and shtuka functions for Drinfeld modules on elliptic curves}, Res. Math. Sci., 5(1):Paper No. 4, 47, 2018.

\bibitem[HJ20]{hartl-juschka}
U. Hartl, A.-K. Juschka, \emph{Pink's theory of Hodge structures and the Hodge conjecture over function fields}, $t$-motives: Hodge structures, transcendence and other motivic aspects, EMS Ser. Congr. Rep., pages 31--182, 2020.

\bibitem[Mau18a]{maurischat}
A. Maurischat, \emph{Periods of t-modules as special values}, Journal of Number Theory, 2018.

\bibitem[Mau18b]{maurischat-prolongations}
A. Maurischat, \emph{Prolongations of $t$-motives and algebraic independence of periods}, Doc. Math., 23:815--838, 2018.

\bibitem[Mat80]{matsumara} H. Matsumara, \emph{Commutative algebra}, 1980, revised and modernized edition by \TeX romencers, available at \href{https://aareyanmanzoor.github.io/assets/matsumura-CA.pdf}{https://aareyanmanzoor.github.io/assets/matsumura-CA.pdf} 2022.

\bibitem[Mil17]{milne} JS. Milne, \emph{Algebraic Groups: The Theory of Group Schemes of Finite Type over a Field}, Cambridge: Cambridge University Press; 2017. 

\bibitem[Pel08]{pellarin08}
F. Pellarin, \emph{Aspects de l'ind\'ependance alg\'ebrique en caract\'eristique non nulle (d'apr\`es Anderson, Brownawell, Denis, Papanikolas, Thakur, Yu, et al.)}, Ast\'erisque, (317):Exp. No. 973, viii, 205--242, 2008.

\bibitem[Pel12]{pellarin}
F. Pellarin, \emph{Values of certain $L$-series in positive characteristic}, Ann. of Math. (2), 176(3):2055--2093, 2012.

\end{thebibliography}
\end{document}